\documentclass [12pt,a4paper,reqno]{amsart}
\usepackage{amssymb}
\usepackage{verbatim}
\usepackage{pgf}
\usepackage{tikz}
\usepackage{needspace}
\usepackage{longtable}
\usepackage[a4paper,margin=1in,includehead,includefoot]{geometry}
\usepackage{nccmath}
\usepackage{enumerate}

\usepackage{amsthm}
\usepackage{amsmath}
\usepackage{stmaryrd}
\usepackage{stackrel}
\usepackage{setspace}
\PassOptionsToPackage{normalem}{ulem}
\usepackage{ulem}
\usepackage{youngtab}
\makeatletter

\numberwithin{figure}{section}
\numberwithin{equation}{section}
\numberwithin{table}{section}

\input amssymb.sty
\usepackage{etoolbox}
\patchcmd{\thebibliography}{\section*}{\section}{}{}

\usepackage{algpseudocode} \usepackage{algorithm}
\usepackage{algpseudocode,algorithm}
\usepackage{lipsum}
\usepackage{caption}
\usepackage{graphics}

\usepackage{amsmath}
\usepackage{amssymb}
\usepackage[all]{xy}

\usepackage{epsfig}\usepackage{youngtab}
	
	\newcommand{\ef}{\end{equation}}
\chardef\bslash=`\\ 

\newdir{:=}{{}}
\newcommand*\colvec[3][]{
	\begin{pmatrix}\ifx\relax#1\relax\else#1\\\fi#2\\#3\end{pmatrix}
}

\newtheorem*{thm*}{Theorem}

\newtheorem{lem}{Lemma}[section]
\newtheorem*{lem*}{Lemma}

\newtheorem*{corl*}{Corollary}

\newtheorem{prop}{Proposition}[section]
\newtheorem{prop*}{Proposition}

\theoremstyle{definition}
\newtheorem{defn}{Definition}[section]
\newtheorem{examp}{Example}

\newtheorem*{examp*}{Example}

\newtheorem*{cor*}{Corollary}
\newtheorem*{remark*}{Remark}
\newtheorem*{CC*}{Crossover Conjecture}
\newtheorem*{Note*}{Note}
\newtheorem*{defn*}{Definition}

\theoremstyle{remark}
\newtheorem{remark}{Remark}[section]

\renewcommand{\sectionmark}[1]{}

\newcommand{\la}{\langle}
\newcommand{\ra}{\rangle}

\newcommand{\defect}{\operatorname{def}}
\newcommand{\hub}{\operatorname{hub}}
\newcommand{\cont}{\operatorname{cont}}

\renewcommand{\a}{\alpha}

\newcommand{\Inv}{\operatorname{Inv}}

\begin{document}
	
		\title[Faces in Crystals]{Faces in crystals of affine type $A$ and the shape of canonical basis elements}
		
		\author{Ola Amara-Omari}
		\address{Department of Mathematics\\
			Bar-Ilan University, Ramat-Gan, Israel}
		\email{olaomari77@hotmail.com}
		
		\author{Ronit Mansour}
		\address{Gordon College of Education 3570503}
		\email{ronitmansour77@gmail.com}
		
		\author{Mary Schaps}
		\address{Department of Mathematics\\
			Bar-Ilan University, Ramat-Gan\\
			Israel}
		\email{mschaps@math.biu.ac.il}
		
		\thanks{Partially supported by the Ministry of Science and Technology fellowship, at Bar-Ilan University, Ramat-Gan, Israel. Some of the results appear in the Ph.D. thesis of the first author.}
		

		\begin{abstract} For a dominant integral weight $\Lambda$ in a Lie algebra of affine type A and rank $e$, and an interval $I_0$ in the residue set $I$, we define the face for the interval $I_0$ to be the subgraph   of the block-reduced crystal $\widehat P(\Lambda)$ that is generated by  $f_i$ for $i \in I_0$.  We show that such a face has an automorphism that preserves defects.
			For an interval of length $2$, we also give a non-recursive construction of the 
			$e$-regular multipartitions with weights in the face, as well as a formula for the number of $e$-regular multipartitions at each vertex of the face.

			For an affine Lie algebra of type $A$ we define and investigate the shape of canonical basis elements, a sequence counting the number of multipartitions with a given coefficient. For finite faces generated by $\Lambda$ with $|I_0|=1,2$, we give a non-recursive closed formula for the canonical basis elements.
		\end{abstract}
		
		\maketitle
		
		\section{INTRODUCTION}
		
		The highest weight module $V(\Lambda)$ for a Lie algebra of affine type A has weight spaces whose elements can be represented either by multipartitions or by canonical basis elements, and we will consider both representations.  The weight spaces of $V(\Lambda)$ for rank $e$ and affine type A have, as representatives, multipartitions that are called  $e$-regular. When we restrict to a face, there are various symmetries among the $e$-regular multipartitions belonging to the face, which we will investigate.  When $I_0=\{1,2\}$ and $\Lambda = a_0\Lambda_0+a_1\Lambda_1+\dots+a_\ell\Lambda_\ell$, we will find paths $2^{u}1^{j_1}2^{j_2-u}$ leading to weights with content $(0,j_1,j_2,0,\dots,0)$ of the face and we will give a complete description of the $e$-regular multipartitions as well as a formula for the number of $e$-regular multipartitions depending only on $a_1,a_2,j_1,j_2$.
		
		We define the shape of a canonical basis element and give an algorithm for constructing the canonical basis element for the case $\lvert I_0 \rvert \leq 2$ which does not require the construction of canonical basis elements for other vertices of the crystal. 
		In an earlier paper on the highest weight modules of an affine Lie algebra \cite{AOS3}, we showed that for any dominant integral weight $\Lambda$ and any non-negative integer $d$ there is a degree $N(d)$ after which every weight of defect $d$ has a path determined by Weyl group reflection with an endpoint of degree  $\leq N(d)$ and such that every intermediate stage of the path is a Weyl group reflection of endpoints.  This had a consequence in \cite{AOS2} regarding the efficient calculation of canonical basis elements. In this paper, we improve the results for canonical basis elements, working with general rank $e$. 
		
		\section{DEFINITIONS AND NOTATION}
		
		Let $\mathfrak g$ be the affine Lie algebra $A^{(1)}_{\ell}$, untwisted of affine type $A$, where $\ell =e-1$ with a Cartan matrix given in the case $e=2$ by
		\[
		C=
		\begin{bmatrix}
			2 &-2\\
			-2&2
		\end{bmatrix}
		\]

		For the general case $e>2$, we have
		\[
		C=
		\begin{bmatrix}
			2&-1&0&\dots&0&-1\\
			-1&2&-1&0&\dots&0\\
			0&-1&2&-1&\dots&0\\
			\dots&\dots&\dots&\dots&\dots&\dots\\
			\dots&\dots&\dots&\dots&\dots&\dots\\
			0&\dots&0&-1&2&-1\\
			-1&0&\dots&0&-1&2
		\end{bmatrix}.
		\]
		
		\subsection{Affine Lie algebras}
		The Cartan subalgebra of the Lie algebra $\mathfrak g$ is generated by \textit{coroots} 
		$h_0, h_1, \dots, h_\ell$ and an element $d$ called the \textit{scaling element},
		which satisfies 
		\[
		\langle d, \alpha_i \rangle =0, 1 \leq i \leq \ell, \langle d,\alpha_0 \rangle =1.
		\]
		
		\noindent The center of  $\mathfrak g$ is one-dimensional, spanned by 
		\[
		c = \sum_{i=0}^\ell h_i,
		\]
		which is called the \textit{canonical central element}.
		The \textit{simple roots} $\alpha_0,\dots, \alpha_\ell$ act on the Cartan subalgebra with an action determined by the Cartan matrix, $\langle h_i,\alpha_j \rangle = a_{ij}$. 
		
		Let $Q$ be the $\mathbb{Z}$-module generated by the simple roots.
		Let  $Q_+$ be the subset of $Q$ in which all coefficients are non-negative.
		The \textit{null root} $\delta$ is $\a_0+\dots+\a_{\ell}$. We define the fundamental weights $\Lambda_j, 0 \leq j \leq \ell$ together with the null root to be weights dual to the coroots and the scaling element, so chosen that $ \langle h_i, \delta \rangle = 0$ for all $i$, $\langle d, \delta \rangle=1$, $ \langle d, \Lambda_j \rangle = 0$ for all $j$, and
		$ \langle h_i, \Lambda_j \rangle = \delta_{ij}$, where $\delta_{ij}$ is the Kronecker delta. 
		
		The $\mathbb Z$-lattice $P$ of weights of the affine Lie algebra spans a real vector space $P_\mathbb{R}$ which has two different bases, one given by the fundamental weights together with the null root, $\Lambda_0,\dots, \Lambda_{\ell}, \delta$, 
		and one is given by $\Lambda_0, \a_0,\dots,\a_{\ell}$.  We will use the first basis for our weights.

		Let $\Lambda$ be a fixed non-zero dominant integral weight, which is a sum of fundamental weights with non-negative integer coefficients. Let $V(\Lambda)$ be a highest weight module with 
		highest weight $\Lambda$, and let $P(\Lambda)$ be the set of weights of $V(\Lambda)$. In affine type A, if $\Lambda=a_0\Lambda_0+a_1\Lambda_1+\dots a_\ell\Lambda_\ell$ for non-negative integers $a_i$, we call $r=a_0+a_1+ \dots +a_\ell$ the \textit{level}.

		We denote by $(\cdot \mid \cdot)$ the symmetric product on the weights, which can be extended by linearity to $P_\mathbb{R}$. The values of this symmetric product of interest to us will be $(\a_i\mid\a_i)=2$, $(\a_i\mid\a_{i\pm 1})=-1$, and in general, $( \a_i \mid \a_j)=a_{ij}$. Also $(\Lambda_i \mid \a_j)=\delta_{ij}$ and  $(\a_i\mid \delta)=0$, for all
		$i,j=0,1,\dots,\ell$. In addition, $(\Lambda_0\,\mid \Lambda_0 )=(\delta \mid \delta) =0$ and $(\Lambda \mid \delta )=r$, the level defined in the previous paragraph.

		\begin{defn}\label{deg} Let $\lambda = \Lambda - \alpha$ with $\alpha \in Q_+$. If $\alpha=\sum_{i=0}^\ell c_i \alpha_i$, where all $c_i$ are non-negative, then the vector 
			\[
			\cont(\lambda)=(c_0.\dots,c_\ell)
			\]
			is called the \textit{content} of $\lambda$. The \textit{degree} $\deg(\lambda)$ of the weight $\lambda$ is the integer $n=\sum_{i=0}^\ell c_i$.
		\end{defn}

		As in [\cite{Kl},\S3.3], we define the defect of a weight $\lambda=\Lambda -\alpha$ by 
		
		\[
		\defect(\lambda)=\frac{1}{2}((\Lambda \mid \Lambda)-  (\lambda \mid \lambda))=(\Lambda \mid \alpha)-\frac{1}{2}(\alpha \mid \alpha).
		\]
		
		\noindent The defect is an integer for the algebras of untwisted affine type $A$ [\cite{Ka}, Prop. 11.4].
		The weights of defect $0$ are those lying in the Weyl group orbit of $\Lambda$. The defect is invariant under the action of the Weyl group.
		
		Every weight $\lambda \in P(\Lambda)$ has the form $\Lambda-\alpha$, for $\alpha \in Q_+$. We will need
		\begin{align*}
			\defect(\lambda-t \delta)=&(\Lambda \mid \alpha+t\delta)-\frac{1}{2}(\alpha +t\delta \mid \alpha +t\delta)\\
			=&\defect(\lambda)+tr.
		\end{align*}

		Define
		\[
		\max (\Lambda)=\{\lambda \in P(\Lambda) \mid \lambda + \delta \not\in P(\Lambda)\},
		\] 
		and by [\cite{Ka}, \S 12.6.1], every element of 
		$P(\Lambda)$ is of the form $\{y-t\delta \mid y \in \max(\Lambda), t \in \mathbb Z_{ \geq 0}\}$. In \cite{BFS} it was shown that there are a finite number of defects of weights in $\max(\Lambda)$, to which all other defects are congruent $\mod r$. 
		Let $W$ denote the Weyl group
		\[
		W=T \rtimes  \mathring{W},
		\] 
		expressed as a semidirect product of 
		a normal abelian subgroup $T$ by the finite Weyl group $\mathring{W}$ given by crossing out the first row and column of the Cartan matrix. [\cite{Ka}, \S 
		6.5.2].

		\begin{defn}\label{hub} For any weight $\lambda$ in the set of weights  $P(\Lambda)$ for a dominant integral weight $\Lambda$, we let $\hub(\lambda)=[\theta_0,\dots, \theta_{\ell}]$ be the \textit{hub} of $\lambda$, where 
			\[
			\theta_i=\la h_i, \lambda \ra.
			\]
			\noindent The hub is the projection of the weight of $\lambda$ onto the subspace of the real weight space generated by the fundamental weights. We write the hub with square brackets to distinguish the hub from the content. The original definition given by Fayers in \cite{Fa} was the negative of this one.
		\end{defn}

		\subsection{Multipartitions}
		
		A partition $\lambda=(\lambda_1,\lambda_2,\dots,\lambda_t)$ is a sequence of integers with $\lambda_1 \geq \lambda_2 \geq \dots \geq \lambda_t$ of length $\ell(\lambda)=t$, and a multipartition is a sequence of partitions.
		
		Let $I$ be the set of residues $\mathbb{Z}/e\mathbb{Z}=\{0,1,\dots,\ell\}$. Choose a sequence of residues  $s=(k_1,\dots,k_r)$, called a \textit{multicharge}, such that $\Lambda=\Lambda_{k_1}+\dots+\Lambda_{k_r}$.  In affine type $A$, the number $r$ of terms in $s$ is called the \textit{level}.
		
		The multicharge determines residues attached to the nodes, as follows.
		For a multipartition $\lambda$ with Young diagram $Y(\lambda)$, the node $(t,u)$ in partition $\ell$ is given residue
		\[
		k_\ell+u-t \mod e.
		\]
		Using these residues, we can define an action of the quantum enveloping algebra
		$\mathcal U = U_v(\widehat{\mathfrak{sl}}(e)) $.
		An addable $i$-node  $\mathfrak n$ is a node of residue $i$  outside $Y(\lambda)$ such that, if added, it would give a multipartition, which we denote by  $\lambda^{\mathfrak n}$.
		A removable $i$-node $\mathfrak m$ inside a multipartition $\mu$ is a node of residue $i$ at the end of a row or column which, if removed,  would give a multipartition, which we denote by  $\mu_{\mathfrak m}$.

		We now define two operators on the multipartitions.
		Our exposition will generally follow that in [\cite{Kl}, \S 3], except for a duality issue which we will explain later.  For any given residue $i$, we denote the addable $i$-nodes, as defined in Def. \ref{Fock}, by a ``$+$'', and the $i$-removable nodes by a ``$-$''. We then write from left to right all the pluses and minuses from the bottom to the top, remove any cases of ``$-+$'', and call the remaining sequence of plus and minus signs the \textit{signature}.  The first removable $i$-node from the left, if such exists, is called \textit{$i$-good},  and the first addable $i$-node from the right, if such exists, is called \textit{$i$-cogood}.
		
		We define an operation of $\tilde e_i$ which is then the removal of the $i$-good node if it exists and otherwise gives $0$, and an operation  $\tilde f_i$, which is the addition of the $i$-cogood node if it exists and otherwise gives $0$.

		The set of all multipartitions obtained by
		acting on the empty multipartition by various $\tilde f_i$, will be called the \textit{$e$-regular multipartitions}. The $e$-regular multipartitions are in one-to-one correspondence with basis elements of a highest weight module $V(\Lambda)$ \cite{Kl}, and the operations $\tilde e_i$ and $\tilde f_i$ correspond to the action of the Chevalley basis elements $e_i$ and $f_i$ of the Lie algebra $\mathfrak g$ on the module $V(\Lambda)$.

		When $r>1$,  there are multipartitions in which every individual partition is $e$-regular but the multipartition as a whole cannot be obtained by this recursive procedure. For example, when $e=3$, and the multicharge is $(0,0,2)$, the multipartition $[\emptyset,(1),\emptyset]$ is not $e$-regular although none of the individual partitions has two identical rows.
		
		There is an analogous construction, usually preferred by Brundan and Kleshchev \cite{BK}[\cite{Kl},\S3.4], which produces multipartitions in which every partition is $e$-restricted, which means that there are no $e$ consecutive columns which are equal.

		\subsection{The block reduced crystal}

		We operate on the set $P(\Lambda)$ of weights of $V(\Lambda)$, using operators $e_i$ and $f	_i$. For a weight $\psi$ in $P(\Lambda)$, we define $f_i \cdot \psi=\psi-\alpha_i$ if the new weight lies in $P(\Lambda)$ and $0$ otherwise, while $e_i$ is defined using $e_i \cdot \psi$ equal to $\psi+\alpha_i$ or $0$.
		\begin{defn}\label{string}
			For any weight $\psi$, the set of weights $\dots, e_i^2(\psi), e_i(\psi),\psi, f_i(\psi), f_i^2(\psi),\dots$ will be called the $i$-\textit{string} of $\psi$.
		\end{defn}
		
		\begin{defn}\label{3.1}
			The set $P(\Lambda)$ can be taken as the set of vertices of a graph $\widehat P(\Lambda)$ which we will call the \textit{block-reduced crystal}. Two vertices will be connected by an edge of residue $i$ if there are two basis elements with those weights connected by $\tilde e_i$ or $\tilde f_i$. A finite set of vertices connected by edges of residue $i$ will be called an $i$-string and all $i$-strings are finite. When convenient, we will identify $\widehat P(\Lambda)$ with its \textit{realization} in $P_\mathbb{R}$, where vertices are represented by points in the real vector space and edges by line segments connecting them.
		\end{defn}

		\subsection{Canonical basis elements}
		
		Our quantum enveloping algebra will be $\mathcal U=  U_v(\widehat{\mathfrak{sl}}(2))$,
		where we are using balanced quantum integers $ [n]_v = v^{n-1}+v^{n-3}+\dots+v^{-(n-3)}+v^{-(n-1)}$.  The quantum factorial is $ [n]_v!= [n]_v \cdot  [n-1]_v \cdot \dots \cdot [1]_v$. The underlying ring of the enveloping algebra is $\mathbb Q(v)$, and the generators are $e_i, f_i, v^{h_i}$  for $i \in I=\mathbb Z/e\mathbb Z$ and a central element $v^d$.

		\begin{defn}\label{Fock} The Fock space $\mathcal F^s$ for multicharge $s$  is a vector space over the field  $\mathbb Q(v)$ with a natural basis corresponding to multipartitions consisting of $r$ partitions, where $\lvert \mu \rangle$ is the natural basis element corresponding to the multipartition $\mu$, and the empty multipartition is represented by $\lvert ~\rangle$.  
			
			The quantum enveloping algebra acts on the Fock space by determining actions for $e_i, f_i, v^{h_i}$ and $v^d$  where $i \in I=\mathbb Z/e\mathbb Z$, as follows:
			\begin{itemize}
				\item  For an $i$-addable node $\mathfrak n$, let us n define  $N(\mathfrak n,i)=\#\{$   addable  $i$-nodes above~$\mathfrak n\}-\# \{$ removable $i$-nodes above $\mathfrak n \}$ and set
				\[
				f_i(\lvert\lambda\rangle)=\sum_{\mathfrak n}v^{N(\mathfrak n,i)}\lvert\lambda^{\mathfrak n}\rangle.
				\]
				
				\item  For an $i$-removable node $\mathfrak m$, let us 
				define $M(\mathfrak m,i)=\#\{$ addable $i$-nodes below~$\mathfrak m\}-  \# \{$ removable $i$-nodes below $\mathfrak m\}$.
				\[
				e_i(\lvert\mu\rangle)=\sum_{\mathfrak m}v^{M(\mathfrak m,i)}\lvert\mu_{\mathfrak m}\rangle.
				\]
				\item Letting $N(i)=\#\{$   addable  $i$-nodes $\}-\# \{$ removable $i$-nodes $\}$, we let $v^{h_i}$ act on an element $\lvert\mu\rangle$  of the natural basis by multiplication by $v^{N(i)}$, giving $v^{N(i)}\lvert \mu\rangle$.
				\item Letting $N_0$ be the number of $0$-nodes in $\mu$, $v^d\lvert \mu \rangle=v^{N_0}\lvert \mu \rangle.$
			\end{itemize}
		\end{defn}
		
		\begin{defn}\label{presig}
			The operator we get by dividing $e_i^k$ or $f_i^k$ by the quantum factorial $[k]_v!$ is called the divided power and will be  represented by
			$e_i^{(k)}$ and $f_i^{(k)}$.

			We define $\mathcal F^s_{\mathcal A}$ to be the subalgebra of $\mathcal F^s$ generated from $\lvert~\rangle$ by the divided powers $f_i^{(k)}$, with coefficients
			in the algebra $\mathcal A$ of Laurent polynomials in $v$ with integral coefficients.
		\end{defn}
		
		There is an involution of the quantum enveloping algebra called the bar-involution which fixes $e_i$, $f_i$, sends $v^{h_i}$ to $v^{-h_i}$, $v^{d}$ to $v^{-d}$, and interchanges $v$ and $v^{-1}$. For each $e$-regular partition $\mu$, there is an element $G(\mu)$ of the Fock space $\mathcal F^s_{\mathcal A}$ that is invariant under the bar involution and is congruent to the natural basis element of $\lvert \mu \rangle$ modulo $v$.   These are called canonical basis elements. An algorithm for constructing the $G(\mu)$ recursively in the case of partitions was given originally by  Lascoux, Leclerc, and Thibon in \cite{LLT}, and later extended to multipartitions by Fayers in \cite{Fa2}.
		
		As mentioned before, we will usually follow Mathas in \cite{M1} in assuming that all partitions with the same corner residue will lie in an interval in the multipartition, and we will also take them in increasing order, $\Lambda=a_0\Lambda_0+a_1\Lambda_1+\dots+a_{e-1}\Lambda_{e-1}$.  The partitions with $i$ in the upper left corner will be called \textit{$i$-corner partitions}.
		
		\begin{defn}\label{dom}
			The \textit{dominance order} on multipartitions is given by
			$\mu \trianglerighteq \lambda$ if, for all  integers $k$ with $1 \leq k \leq r$ and $j \leq \ell(\mu^k)$,
			\[
			\sum_{\ell=1}^{k-1}\mid \mu^\ell \mid + \sum_{i=1}^j \mu^k_i \geq \sum_{\ell=1}^{k-1}\mid \lambda^\ell \mid + \sum_{i=1}^j \lambda^k_i.
			\]	
		\end{defn}

		\begin{defn}
			For an $e$-regular multipartition $\mu$, we can write the canonical basis element $G(\mu)$ in the natural basis.  If $\lvert \lambda \rangle$ is a natural basis element with non-zero coefficient, then $\mu$ and $\lambda$ have the same content.  Furthermore, in the dominance relation, Def. \ref{dom}, $\mu \trianglerighteq \lambda$.
		\end{defn}

		\section{SLICES AND FACES}
		
		Let $I_0$ be a proper subinterval of the set of residues $I = \{0,1,\dots,\ell\} \subset \mathbb Z/e \mathbb Z$. 
		\begin{defn} We give a realization of the graph $\widehat P(\Lambda)$ by embedding it in $P_\mathbb{R}$.
			An $I_0$  \textit{slice} is the intersection of $\widehat P(\Lambda)$ with a linear subspace on which all components of the content for residues outside of $I_0$ are fixed. A \textit{face} is a slice containing the highest weight $\Lambda$ so that $c_j=0$ for $j \notin I_0$. 
		\end{defn}	
		\begin{examp} If $I_0 = \{i\}$, then an $I_0$ slice is an $i$-string	.
			If $I_0=I-\{i\}$, a slice consists of all vertices with a fixed $c_i$. In \cite{AS} this was called a floor.
		\end{examp}

		Every face has an involution in its hub coordinates which takes the lowest degree term, corresponding to $\Lambda$, to the highest degree term, which we will denote by $\rho$, and we will show that the involution preserves defects.
		This is a consequence of the duality of an irreducible module for a finite Lie algebra, as in [\cite{H},\S 10, \S 21]
		or [\cite{B}, Chapter VIII, Section 5, Proposition11(ii)]. 
		Specifically, there is a duality between the highest weight module $V(\lambda)$ for the highest weight and $V(-\sigma \lambda)$ where $\sigma$ is the longest element of the finite Weyl group. The procedure uses the automorphism of the Dynkin diagram. The approach in \cite{H} is much closer to our treatment, but the actual results we need are given as exercises, 10.9 and 21.6.  In our case, the automorphism reverses the order of $\{1,2,\dots,t\}$ giving exactly the element $\rho$ we calculate in the next lemma.  However, since we are dealing with a linear subspace inside an infinite module, we also calculated the components of the hub outside of $I_0=\{i,i+1,\dots,i+t-1\}$.

		For the  calculation of the face, we need to know only the coefficients $a_i,\dots, a_{i+t-1}$ of  
		$\Lambda$. To simplify the notation, we assume $i=1$ and $e \geq  t+2$.  The cyclic reordering of the indices is allowable in type A because of the symmetry of the Cartan matrix.  We surely have $e \geq t+1$ and will treat the case of equality as a consequence of the more general case.
		
		\begin{lem} Suppose $t$ is a positive integer, the rank $e$ of a Kac-Moody algebra of affine type A is $e \geq t+2$, and $I_0= \{1,\dots,t\}$. If 
			$\Lambda=a_0\Lambda_0+a_1\Lambda_1+\dots+a_\ell\Lambda_\ell$ with 
			$[a_0,a_1,\dots,a_\ell]$ and $r'=a_1+\dots+a_t$, then the  
			$I_0$ face has an element $\rho$ with  $\hub(\rho)=[a_0+r',-a_t,\dots,-a_1,a_{t+1} +r',a_{t+2},\dots]$.
		\end{lem}
		\begin{proof}
			We have to show that there is a path within the face constructed from Weyl group reflections with $\rho$ as its endpoint. We want to work by induction on $t$.  If $t=1$, the necessary path is simply acting by $s_1$, which is $f_1^{a_1}$, In that case $r'=a_1$ and the new hub is $[a_0+a_1, -a_1,a_1+a_2,a_3,\dots]$. Now suppose $t=2$.  We begin with hub
			$[a_0,a_1,a_2,a_3,\dots]$.  We can act by $s_1$, which is $f_1^{a_1}$, $s_2$ which is $f_2^{a_1+a_2}$, and finally  times $f_1^{a_2}$ again.  The hubs along the path are 
			\begin{align*}
				[&a_0, &a_1, &a_2, &a_3, &a_4, &\dots]\\
				[&a_0+a_1, &-a_1, &a_1+a_2, &a_3, &a_4, &\dots]\\
				[&a_0+a_1, &a_2, &-a_1-a_2, &a_1+a_2+a_3, &a_4, &\dots]\\
				[&a_0+a_1+a_2, &-a_2, &-a_1, &a_1+a_2+a_3, &a_4, &\dots]\\	
			\end{align*}
			We could, equally well, have started with $a_2$ copies of $f_2$, and the result would have been the same.
			
			We now do induction on $t \geq 3$, assuming that the lemma is true for $t-2$. 
			The highest weight element has, as usual, the hub $[a_0,a_1,\dots,a_\ell]$ and 
			defect $0$.  We let $r=a_0+\dots+a_\ell$ be the level, as usual, but now we also define a \textit{local level} $r'=a_1+a_2+\dots+a_t$. We generate a sequence of weights of defect $0$ by  starting with the Weyl group generator $s_1$, followed by 
			$s_2, s_3,\dots, s_t$. The hubs we get are
			\begin{align*}
				[&a_0,&a_1,&a_2, &a_3,&\dots,&a_t,&a_{t+1},&\dots]\\
				[&a_0+a_1,&-a_1, &a_1+a_2,&a_3,&\dots,&a_t, &a_{t+1},&\dots]\\
				[&a_0+a_1,&a_2, &-a_1-a_2,& a_1+a_2+a_3,&\dots,&a_t, &a_{t+1},&\dots]\\
				[&\dots&\dots&\dots &\dots&\dots&\dots &\dots&\dots]\\
				[&a_0+a_1,&a_2, &a_3,& a_4,&\dots,&-r', &r'+a_{t+1},&\dots]\\
			\end{align*} 
			\noindent We now work back the other direction, using $s_{t-1}, s_{t-2},\dots,s_1$, giving
			\begin{align*}
				[&a_0+a_1,&a_2, &a_3, & a_4,&\dots, &-r', &r'+a_{t+1},&\dots]\\
				[&a_0+a_1,&a_2, &\dots, &a_{t-1}+a_t, &-a_t,&a_t-r', &r'+a_{t+1},&\dots]\\
				[&\dots&\dots&\dots &\dots&\dots&\dots &\dots&\dots]\\
				[&a_0+a_1,&a_2+a_3+\dots+a_t, &-a_3-a_4-\dots-a_t, &a_3\dots,&a_{t-1}, &a_t-r', &r'+a_{t+1},&\dots]\\
				[&a_0+r',&a_1-r', &a_2, &a_3\dots,&a_{t-1}, &a_t-r', &r'+a_{t+1},&\dots]\\
			\end{align*} 
			\noindent We have a new local level $r''=a_2+\dots+a_{t-1}$. 
			By the induction hypothesis, we have a path involving only $I_0'=\{2,\dots,t-1\}$ from the hub $h$ to a weight with hub $[a_0+r',a_1-r'+r'',-a_{t-1},\dots,-a_2,a_t-r'+r'', a_{t+1}+r', a_{t+2},\dots]$. There was nothing in the procedures we followed which required the two components on either side of the interval to be positive.   On the left side, $a_1-r'+r''-a_2-a_3+\dots-a_t+r''=-a_t$. On the right side, $a_{t}-r'+r''=-a_1-a_2-\dots-a_{t-1}+r''=-a_1.$ The final result is $[a_0+r',-a_t,\dots,-a_1,a_{t+1}+r',a_{t+2},\dots]$ as required. 
		\end{proof}

		Following the same succession of actions of $s_i$, we can calculate the multipartition corresponding to $\rho$.  The first sweep from $1$ to $t$ fills in
		the first row of each partition, and the return from $t-1$ to $1$ fills in the first column. The $1$-corner  and $t$-corner partitions are now finished, either a row of length $t$ or a column of length $t$. We repeat the procedure but truncated, using 
		$s_2, s_3,\dots_{t-2}, s_{t-1},s_{t-2},\dots,s_2$.
		The passage from $2$ to $t-1$ and back to $2$ will fill in the second row and second column.  We follow this by $s_3,\dots,s_{t-2}, \dots,s_3$, continually truncating until the center is reached. Each partition will be a rectangle
		whose first row and column form a hook of length $t$. If $t$ is even, equal to 2c, the total degree will then be 
		\[
		\deg(\rho)=(a_1+a_t)t+(a_2+a_{t-1})2(t-1)+(a_3+a_{t-2})3(t-2)+\dots+(a_{c}+a_{c+1})c(c+1).
		\]  
		\noindent If $t=2c-1$, the last term in the sum is replaced by $(a_c)c^2$. 
		Since we have constructed the longest element in the finite Weyl group, this is its length.
		\begin{remark}
			We have done all our calculations for cases with $e \geq t+2$.  When 
			$e=t+1$, we add together the $0$ and $t+1$ components of the hub, since 
			in that case $\alpha_1$ and $\alpha_t$ both contribute to that component.	
		\end{remark}
		
		\begin{examp}We construct a face of $\widehat P(\Lambda)$ for                                                                                                                                                                                                                                                                                                           $e=5$,$\Lambda=\Lambda_1+\Lambda_2$,$I_0=\{1,2,3\}$, giving hubs and defects for most vertices, with the defect written as a superscript on the hub. Note the symmetry in the crystal. 
			
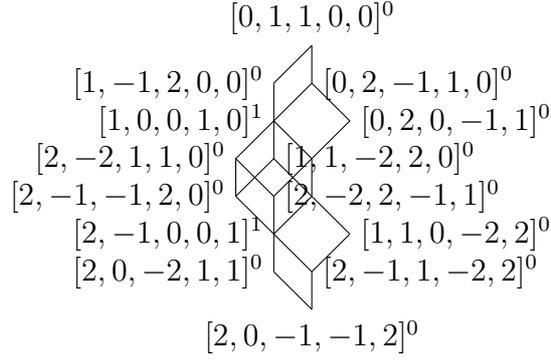
\begin{center} 
				\begin{tikzpicture}
					\draw (0,0) node[anchor=south]{$[0,1,1,0,0]^0$} -- (-.5, -.5) node[anchor=east]{$[1,-1,2,0,0]^0$};
					\draw (0,0) -- (0,-.5) node[anchor=west]{$[0,2,-1,1,0]^0$} -- (-.5,-1) node[anchor=east]{$[1,0,0,1,0]^1$}--  (-1, -1.5) node[anchor=east]{$[2,-2,1,1,0]^0$} --
					(-1, -2)node[anchor=east]{$[2,-1,-1,2,0]^0$};
					\draw (0,-.5) --(.5,-1) node[anchor=west]{$[0,2,0,-1,1]^0$}
					--(-.5,-2);
					\draw (-.5,-.5)--(-.5,-1.5) node[anchor=west]{$[1,1,-2,2,0]^0$} --
					(.5,-2.5) node[anchor=west]{$[1,1,0,-2,2]^0$};
					(-.5,-2);
					\draw (0,-1.5) -- (0,-2);
					\draw (-.5,-1.5) -- (-1, -2) ; 
					\draw (-.5,-1) -- (0,-1.5);
					\draw (-1,-1.5) -- (-.5,-2) node[anchor=west]{$[2,-2,2,-1,1]^0$};
					\draw (-1,-2) -- (-.5,-2.5)node[anchor=east]{$[2,-1,0,0,1]^1$}--(0,-3)node[anchor=west]{$[2,-1,1,-2,2]^0$}--
					(0,-3.5) node[anchor=north]{$[2,0,-1,-1,2]^0$};
					\draw  (-.5, -2) -- (-.5,-3)node[anchor=east]{$[2,0,-2,1,1]^0$};
					\draw (-.5,-3)  -- (0,-3.5);
					\draw  (0, -2)  -- (-.5,-2.5);
					\draw (.5,-2.5)  -- (0,-3);
				\end{tikzpicture}
			\end{center}
			
			\noindent The Young diagram for the unique multipartition with weight $\rho$, with residues from $I_0$ inserted, is given by\\ 
			\noindent \young(123)\\
			\\
			\young(23,12)
		\end{examp}
		We need the following lemma:
		
		\begin{lem}\label{defect}[\cite{AOS2}, Lem. 3.1]For a dominant integral weight
			$\Lambda=a_0\Lambda_0+a_1\Lambda_1+\dots+a_\ell \Lambda_{\ell}$, if $\eta=\Lambda-\alpha$ for $\alpha \in Q_+$ has a non-negative $i$-component $w=\theta_i$ in $\hub(\eta)$, so that $\lambda$ lies at the high weight end of the $i$-string, then
			\begin{enumerate}
				\item   The defect of $\lambda=\eta-k\alpha_i$ for $ 0 \leq k \leq w$ is
				$	\defect(\eta)+k(w-k)$.
				\item  The $i$-components of the hub descend by $2$ as we move down the $i$-string,  so the absolute values of the hub component $\theta_i=\langle \lambda, h_i \rangle$ decrease as the defect increases and then increase as the defect decreases.
			\end{enumerate}	
			In summary,
			\begin{center}
				\begin{tabular}{ | l || c | c | c | c | c | c | }
					\hline
					$\lambda$ &$\eta$ & $\eta-\a_i$&$\dots$ & $\eta-k\a_i$ &$\dots$ & $\eta-w\a$ \\ \hline \hline
					$\defect(\eta)$ &$\defect(\eta)$&$\defect(\eta)+(w-1)$&$\dots$&$\defect(\eta)+k(w-k)$&$\dots$ & $\defect(\eta)$\\ \hline
					$\langle \lambda ,h_i \rangle$ &w&w-2&\dots&w-2k&\dots&-w \\ \hline					
					\hline
				\end{tabular}
			\end{center}
		\end{lem}
		
		The vertices of a face are essentially a block-reduced crystal of a finite Lie algebra of type A,
		and as such, have a block-reduced version of the  Bump-Schilling involution \cite{BS}.
		
		\begin{prop}\label{sym}
			A face for an interval $\{1,2,\dots,t\}$
			has an involution which is a reflection in the vertices of degree $r=a_1+a_2+\dots + a_t$. The involution $\tau$ sends a weight $\lambda$ with hub 	$h=[b_0,b_1,\dots,b_{t+1},b_{t+2},\dots]$  to a weight $\lambda'=\tau(\lambda)$ with hub
			$h'=[a_0+r-(b_{t+1}-a_{t+1}),-b_t,\dots,-b_1,a_{t+1}+r-(b_0-a_0),b_{t+1},\dots]$ and satifies    $(\tau(f_i\lambda))=e_{t+1-i}\tau(\lambda)$ and $(f_i\tau(\lambda'))=\tau(e_{t+1-i}(\lambda'))$. The involution $\tau$ preserves the defect of a weight.
		\end{prop}
		\begin{proof}
			We first note that, on a face, the coefficient of the null root $\delta$ is zero. Therefore, the hub determines the weight.
			On the components from $1$ to $t$, the map $\tau$ is the composition of two involutions, reversing the order and multiplication by $-1$.  We will check
			that $\tau$ is an involution on the $0$ component of the hub, and the $t+1$ component is dual.
			\begin{align*}
				\tau(\tau([b_0,b_1,b_2,\dots]))=&\tau[a_0+r-(b_{t+1}-a_{t+1}),-b_t,-b_{t-1},\dots]\\
				=&[a_0+r-((a_{t+1}+r-(b_0-a_0))-a_{t+1}),b_1,b_2,\dots]\\
				=&[b_0,b_1,b_2,\dots]
			\end{align*}
			
			We use induction on the degree, which is equivalent to induction on the length of a path in the crystal leading to $\lambda$. For the base case, $\tau(\Lambda)=\rho$ since in that case $b_0=a_0$ and $b_{t+1}=a_{t+1}$. Let us assume that the lemma holds for a path of length $m$, and let us act by some $f_i$ to get $\bar \lambda$, which means subtracting $\alpha_i$.
			The dual operation is acting on $\tau(\lambda)$by $e_{t+1-i}$, and the effect on the hub is to add $\alpha_{t+1-i}$ to $\hub(\tau(\lambda))$. It suffices the check on the hubs that
			$f_i(\tau(\lambda))=\tau(e_{t+1-i}\lambda)$. 
			In the $i$-th component we get $b_i-2$, and in the dual, in the $(t+1-i)$-component, we have $-b_i+2$ as needed.  If $j$ is adjacent to $i$ but not in the set 
			$\{0, t+1$\}, then the corresponding values are $b_j+1$ and $-b_j-1$.
			If $i$=1, then the $0$ component changes from $b_0$ to $b_0+1$, 
			while the $t+1$ component of the dual changes from 
			$a_t+r-(b_{0}-a_{0})$ to 	$a_t+r-(b_{0}+1-a_{0})$, which gives
			$\tau(e_{t+1-i}\lambda)$ as needed.	The case $i=t$ is dual.
			
			To show that defects are preserved, we use the same path and its dual. Since by Lemma \ref{defect}, the defect along a string depends only on the length of the string and the defect at one end, the duality in the form of the crystal indicates equality for the defects. 
		\end{proof}
		
		\begin{examp} We illustrate Prop. \ref{sym} Consider a dominant integral weight $\Lambda=a_0\Lambda_0+\dots+a_\ell\Lambda_\ell$. Set $r'=a_1+a_{2}$. The crystal that represents the
			module generated from $\Lambda$ by $f_1,f_{2}$ has a hexagonal form and highest degree $2r'$
			
			The corners of the hexagon will be defined  by the defect $0$ elements
			$\Lambda$, $s_1\Lambda$, $s_2\Lambda$,  $s_2s_1\Lambda$, $s_1s_2\Lambda$, $s_2s_1s_2\Lambda$= $s_1s_2s_1\Lambda$. We calculate the hubs on one of the two paths leading to the element of the highest degree since the degree is the sum of the number of copies of $\alpha_1$ and $\alpha_2$ which we subtract.  Acting by $a_1$ copies of $\alpha_1$ give a new hub
			$[\dots,a_{0}+a_1,-a_1,a_1+a_2, a_3,\dots]$.
			Since the component at $2$ is $a_1+a_2$, we subtract $a_2$ copies of $\alpha_2$.  This gives a hub
			\[
			[\dots, a_{0}+a_1,a_2,-a_1-a_2, r'+ a_3,\dots]
			\]
			As a final step, we remove $a_2$ copies of $\alpha_1$, giving a hub
			\[
			[\dots, a_{0}+r',-a_{2},-a_1, r'+ a_{3},\dots]
			\]	
			
			The hexagon has an axis of symmetry given by the hyperplane of points of degree $r'$. We claim that symmetrically located points have identical defects and thus canonical basis elements for those defects have the same degree as polynomials.
			Since the $1$ and $2$ components of the hub are negative, we can go no farther with only $f_1$ and $f_{2}$.
			
			We consider the particular case $a_1=3,a_2=2$, for which we will draw the hexagon.  Since the content is $0$ except in components $1$ and $2$, we abbreviate the content by $(j_1,j_2)$ and label the external vertices by their content, with the defect as a superscript. The remaining vertices will be labeled by their defect to illustrate the subsequent lemma, Lemma \ref{def}.
			
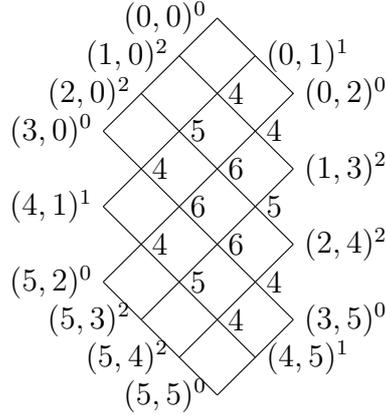
			\begin{center} 
				\begin{tikzpicture}
					\draw (0,0) node[anchor=east]{$(0,0)^0$} -- (-.5, -.5) node[anchor=east]{$(1,0)^2$}--
					(-1, -1) node[anchor=east]{$(2,0)^2$}--(-1.5, -1.5) node[anchor=east]{$(3,0)^0$};
					\draw (0,0) -- (.5,-.5) node[anchor=west]{$(0,1)^1$} -- (1,-1) node[anchor=west]{$(0,2)^0$} -- (-1.5,-3.5) ;
					\draw (-.5,-.5) --(0,-1)node[anchor=west]{$4$} --(.5,-1.5) node[anchor=west]{$4$}
					--(1,-2)node[anchor=west]{$(1,3)^2$};
					\draw (.5,-.5)--(-1.5,-2.5)node[anchor=east]{$(4,1)^1$};
					\draw (-1,-1) -- (-.5,-1.5)node[anchor=west]{$5$}
					-- (0,-2)node[anchor=west]{$6$}--(.5,-2.5)
					node[anchor=west]{$5$}--(1,-3)node[anchor=west]{$(2,4)^2$};
					\draw (1,-2)--(-1,-4);
					\draw (-1.5,-1.5) -- (-1,-2)node[anchor=west]{$4$}
					-- (-.5,-2.5)node[anchor=west]{$6$}--(0,-3)
					node[anchor=west]{$6$}--(0.5,-3.5)
					node[anchor=west]{$4$}--(1,-4)node[anchor=west]{$(3,5)^0$};
					\draw (1,-3)--(-.5,-4.5);
					\draw  (-1.5,-2.5)--(-1,-3)node[anchor=west]{$4$}--(-.5,-3.5)node[anchor=west]{$5$}--(0,-4)node[anchor=west]{$4$}--(.5,-4.5)node[anchor=west]{$(4,5)^1$};
					\draw (-1.5,-3.5)node[anchor=east]{$(5,2)^0$}--(-1,-4) node[anchor=east]{$(5,3)^2$}--(-.5,-4.5)node[anchor=east]{$(5,4)^2$}--(0,-5)node[anchor=east]{$(5,5)^0$};
					\draw (1,-4)--(0,-5);
				\end{tikzpicture}
			\end{center}

		\end{examp}

		\begin{lem}\label{def}
			For a face with only operation by $f_1$ and $f_2$,  there is a weight with content $(j_1,j_2)$ if $0 \leq j_1 \leq a_1+j_2$ when $j_2 \leq a_2$.
			or $0 \leq j_2 \leq a_2+j_1$ when $j_1 \leq a_1$ or $0 \leq j_1,j_2 \leq a_1+a_2$. The defect of a weight with content $(j_1,j_2)$, $0 \leq j_1 \leq a_1$ and $0 \leq j_2 \leq a_2$ is 
			\[
			j_1(a_1-j_1)+j_2(a_2-j_2)+j_1j_2,
			\]
		\end{lem}
		\noindent Most other defects in the face can be obtained by duality or 
		by reflection in an $i$-string for some $i$. 
		
		\begin{proof}
			We apply the lemma above, going down the $2$-string to $j_2$ and then going across the $1$-string, which will have length $a_1+j_2$, giving altogether 
			\[  
			j_2(a_2-j_2)+j_1(a_1+j_2-j_1).
			\]	 
		\end{proof}

		\begin{prop}\label{e-reg} Let $\Lambda=a_0\Lambda_0+a_1\Lambda_1+a_2\Lambda_2+ \dots a_\ell\Lambda_\ell$ be a dominant integral weight in a Lie algebra of affine type A and rank $e$ for $e \geq 3$.
			\begin{enumerate}[(i)]
				\item
				In the  face of $\widehat{P}(\Lambda)$ for $I_0=\{1,2\}$, 
				the $e$-regular multipartitions $\mu$ with content $(j_1,j_2)$ 
				are determined by setting $t=\min(a_1,j_1)$, letting $x= \max(0,j_1-a_1)$ and choosing
				$w$ with $\max(0,j_2-a_2) \leq w \leq \min(t,j_2, a_1+j_2-j_1) $, where the $a_1$ $1$-corner partitions consist of $w$ partitions $(2)$, $t-w$ partitions $(1)$, $a_1-t$ partitions $\emptyset$, and while the $a_2$ $2$-corner partitions  consist of $x$ partitions $(1,1)$, $j_2-w-x$ partitions $(1)$ and finally $a_2-j_2+w$ partitions $\emptyset$. Each $e$-regular multipartition is reached by a path $p(u)=2^u1^{j_1}2^{j_2-u}$, where $u=\min(j_1-w,j_2-w)$. 
				\item Letting $\bar{j_1} =a_1+\min(a_2,j_2)-j_1$ and $\bar j_2=a_2+\min(a_1,j_1)-j_2$, the number of $e$-regular multipartitions with content $(j_1,j_2)$ is $\min(j_1,j_2,a_1,\bar j_1)+1$ when $j_2 \leq a_2$, and $\min(\bar{j_2},a_2,\bar{j_1})+1$ when $j_2>a_2$.
			\end{enumerate}
		\end{prop} 
		\begin{proof} \begin{enumerate}[(i)]
				\item The restrictions on $t$, $x$, and $w$ result from the natural requirement that each type of partition cannot occur a negative number of times. The total number $t+x$ of added $1$-nodes must be $j_1$ so $t \leq j_1$ always holds, and, similarly, we always have  $w \leq j_2$.  Since when building an $e$-regular multipartition, we add all $1$-corner nodes from top to bottom without gaps, either $t=j_1 \leq a_1$ and $x=0$ or $t=a_1$ and $x=j_1-a_1$. We now consider the non-negativity of the six different possible kinds of partitions listed in the proposition, together with their implications.
				
				\begin{itemize}
					\item  $w \geq 0$.
					\item  $t-w \geq 0$ implies that $w \leq t$.
					\item  $a_1-t \geq 0$ implies $t \leq a_1$.
					\item  $x \geq 0$.
					\item  $j_2 -w-x \geq 0$ implies that $w \leq j_2-x \leq j_2+a_1-j_1$.
					\item  $a_2-(j_2-w) \geq 0$ implies that $w \geq j_2-a_2$. 
				\end{itemize}
				Combining all this information gives the conditions in the proposition:
				$t=\min(j_1, a_1)$, $x=\max (0, j_1-a_1)$, and $\max(0, j_2-a_2) \leq w \leq \min(t, j_2, a_1+j_2-j_1)$.

				We proceed to prove (1) by induction on $j_1+j_2$.  For $j_1+j_2=0$, there is the unique empty multipartition.  As an induction step, we assume that the proposition holds for 
				$j_1+j_2-1$.  For a given content $(j_1,j_2)$, there is a unique $t$ and $x$, but a range of choices for $w$.  We must show that for every choice of $w$ satisfying the bounds, we get an $e$-regular multipartition by following the given path, and then that, given an $e$-regular multipartition, it must correspond to some $w$ in the given range and be reached by the given path. Every $e$-regular multipartition $\mu$ with content $(j_1,j_2)$ is obtained from $\mu'$ of content $(j_1-1,j_2)$ or $(j_1,j_2-1)$, and we will use the induction hypothesis to find the new $w$ and to show that the new path leads to $\mu$.
				
				We start by choosing $w$ satisfying $\max(0,j_2-a_2) \leq w \leq \min(t,j_2, a_1+j_2-j_1)$.
				We will first need to show that the given choice of $u=\min(j_1-w,j_2-w)$ gives a path that lies in the face. Since $w \geq j_2-a_2$, we have 
				$u =\min(j_1-w,j_2-w) \leq j_2-w \leq a_2$, so the first segment lies in the face.
				The $1$-string from $u$ has length $a_1+u$, so portion of the $1$-string of length $j_1$ surely lies in the face if $j_1 \leq a_1$, and if $j_1 > a_1$, then $w \leq t \leq a_1$, which implies that $ j_1-w \geq j_1-a_1$. . We have chosen $u=\min(j_1-w,j_2-w)$. If $u=j_1-w$ then $a_1+u \geq j_1$ so the path $2^u1^{j_1}$ lies in the crystal. If $u=j_2-w$, then since $w \leq a_1+j_2-j_1$, we have $j_1 \leq a_1+j_2-w=a_1+u$, and again, the path $2^u1^{j_1}$ lies in the crystal. The $2$-string is convex and the endpoint with content $(j_1,j_2)$ lies in the face, so the entire segment lies in the face. 
				
				We now show that choosing the $i$-cogood node at each step of the path produces a multipartition with $w$ copies of the partition $(2)$ at the beginning.  The first $u$ $2$-nodes produce partitions $(1)$ in the first $u$ $2$-corner nodes. There are no removable $1$-nodes, so the $j_1$ addable $1$-nodes go to produce $t$ partitions $(1)$, and, if $j_1>a_1$, another $x$ partitions $(1,1)$.  By definition,  $u=\min(j_1-w,j_2-w)$.
				If $u=j_1-w \leq j_2-w$, then substituting $j_1=t+x$ gives $u-x=t-w$. Thus there are exactly $w$ of the $j_2-u$ nodes we must add which are not canceled by the $u-x$ removable nodes, giving a multipartition whose first $w$ partitions are $(2)$.  If $u=j_2-w \leq j_1-w$, then $w=j_2-u$, the number of $2$-nodes to be added in the last segment of the path, and the removable nodes do not interfere because $u-x \leq t-w$. When we add the remaining $j_2-u$ $2$-nodes, we get $\mu$.

				We have shown that any $w$ satisfying the inequalities in the proposition corresponds to an $e$-regular multipartition, and now we must show that any $e$-regular multipartition comes from such a $w$ and is reached by the path $p(u)$. In the first paragraph of this proof, we showed that the conditions on $t,x,w$ are consequences of the structure of the multipartition as given in the proposition. We need to show that the path $p(u)=2^u1^{j_1}2^{j_2-u}$ lies in the crystal, where $u=\min(j_1-w,j_2-w)$, and ends at the multipartition $\mu$.  We are using induction, assuming the $\mu$ comes from an $e$-regular multipartitions $\mu'$ for which the number $w'$ of partitions $(2)$ satisfies the inequalities
				$\max(0,j_2'-a_2) \leq w' \leq \min(t',j_2', a_1+j_2'-j_1') $
				We now consider two cases.
				\begin{enumerate}
					\item If $j_1$ is fixed, so that $\mu'$ has content $(j_1,j_2')$, and $j_2=j_2'+1$. Neither $t$ nor $x$ will be changed by adding a $2$-node, so we have $t'=t$ and $x'=x$. The $2$-signature of $\mu'$  has $w'$ removable $2$-nodes, $t-w'$ addable $2$-nodes, $j_2'-w'-x$ removable $2$-nodes and $a_2-j_2'+w'$ addable $2$-nodes.
					\begin{itemize}
						\item Suppose $j_1>j_2'$ so that $j_1-w'>j_2'-w'$, giving 
						$u'=\min(j_1-w',j_2'-w')=j_2'-w'<j_1-w'$.  Using $t+x=j_1$, the condition becomes $t-w' > u-x$, that is, there are more $2$-addable nodes than can be canceled out by the removable nodes. Adding the $1$-corner $2$-cogood node makes $j_2=j_2'+1,w=w'+1,u=u'$, and indeed, we reached this $e$-regular multipartition by the path  $2^u1^{j_1}2^{j_2-u}$. Since by induction, we had $j_1' \leq a_1+u'$, this is still true with $j_1'=j_1$ and $u'=u$, so the path to $\mu$ lies in the face. 
						\item  Now suppose that $j_1 \leq j_2'$ so that $u'=j_1-w'$. Then the number of removable $2$-corner nodes in the signature exactly cancels the addable $1$-corner nodes, so the $2$-cogood node must be the first addable $2$-corner node in position $j_2'-w'+1$.  Thus $w=w'$, $u=j_1-w'=u'$ and we have again reached 
						$\mu$ by the path $2^u1^{j_1}2^{j_2-u}$.
					\end{itemize}

					\item If $j_2$ is fixed, we are considering the case that we have $\mu'$ to which we can add a $1$-cogood node to get $\mu$. Since we add no $2$-node, $w=w'$.  We divide into two cases, $j_1'<a_1$ or $j_1' \geq a_1$.
					\begin{itemize}
						\item 	In case $j_1'<a_1$, then $t'=j_1'$ and $x'=0$. We assume that there is a  $1$-cogood node, and adding it sets $t=t'+1$.  By definition, 
						$u'=\min(j_1'-w',j_2'-w').$ 
						
						If $j_1'-w'<j_2-w'$ then $j_1-w \leq j_2-w$, so $u=j_1-w=u'+1$. The change in $u$ allows us to reach $\mu$ by a path $2^u1^{j_1}2^{j_2-u}$ . Since $t=j_1$ and 
						$u=j_1-w$, the $u$ removable nodes exactly cancel all the addable $2$-nodes
						except the first $w$, and the remaining $2$-nodes are then added after the first $u$ $2$-corner nodes.
						
						If $j_1'-w'\geq j_2-w'$, then $u'=j_2-w'$ so $u=u'$. In this case, too, we can reach $\mu$ by the path $2^u1^{j_1}2^{j_2-u}$ because, when we have added the 
						cogood nodes of the path $2^u1^{j_1}$, we have only 
						$w=j_2-u$ cogood $2$-nodes to add and this gives us $\mu$.
						
						\item In the case $j_1' \geq a_1$, then $t'=t=a_1$. Since there is an addable $1$-node, we must have
						$u'>x'=j_1'-a_1$, and after it is added we will have $x=x'+1$. Furthermore, by adding $a_1$ to both sides of the inequality,
						we get $u'+a_1 > j_1'$, so that $u'+a_1\geq j_1'+1=j_1$. since by its definition, $u \geq u'$, this shows that $2^u1^{j_1}$ lies in the face. Since the face is convex and $(j_1,j_2)$ is the content of a vertex in the face, the entire path $2^u1^{j_1}2^{j_2-u}$ lies in the face. It remains to finish the proof that choosing the $i$-cogood nodes at each vertex along the path produces the desired multipartition $\mu$.
						
						If $u'=j_2-w$, we set $u=u'$ since there are no more $2$-nodes with that value of $w=w'$. Since $u' =\min(j_1-w,j_2-w) \leq j_1'-w'$, and $j_1'=t'+x'$, we have
						$u'-x'\leq t'-w'$ and thus $u-x <t-w$. If we now use the path $2^u1^{j_1}$, we will have exactly $w=j_2-u$ $2$-nodes to add,
						and the $u-x$ removable $2$-nodes will not interfere with adding them.
						
						If $u'<j_2-w$, then $u'=j_1'-w$. We then have $u=u'+1=\min(j_1+1-w,j_2-w)$, so $u=j_1-w$.  Using 
						$j_1=t+x$, we have $u-x=t-w$, so the number $u-x$ of removable $2$-nodes equals the number $t-w$ of addable $1$-corner $2$-nodes.
						Thus we can reach $\mu$ by taking cogood nodes along the path $2^u1^{j_1}2^{j_2-u}$.
						
					\end{itemize}
					
				\end{enumerate}
				\item	Since we have shown that each $w$ satisfying the given inequalities, and only those $w$, give $e$-regular multipartitions, we need only add $1$ to the difference between the upper bound $U=\min(t,j_2,a_1+j_2-j_1)$ and the lower bound $L=\max(0, j_2-a_2)$ to find the number of $e$-regular partitions with content $(j_1,j_2)$. We defined numbers $\bar{j_1}$ and $\bar{j_2}$, which calculate the distance of $(j_1,j_2)$ from the bottom of its $i$-string for $i=1,2$. Since $t=\min(j_1,a_1)$, by using $\min(\min(a,b),c,d)=\min(a,b,c,d)$, we get $U=\min(j_1,a_1,j_2,a_1+j_2-j_1)$, as needed.
				
				The face divides naturally into two sectors:
				\begin{enumerate}
					\item $ j_2 \leq a_2$. Since $j_2-a_2 \leq 0$, the lower bound $L$ is $0$, and the upper bound is $U=\min(j_1, a_1, j_2, \bar{j_1})$, so $U-L= \min(j_1,j_2,a_1,\bar{j_1})$.
					
					\item If $j_2 > a_2$, then $L=j_2-a_2$. We have $U=\min(\min(a_1,j_1), j_2, a_1+j_2-j_1)$. We subtract $j_2-a_2$ from each of the three quantities in $U$. Using  $\bar{j_2}  = a_2 +\min(a_1,j_1)-j_2 $, the first term becomes $\bar{j_2}$. For the second term, $j_2-(j_2-a_2)=a_2$. For the third term, subtracting $j_2-a_2$ produces $a_1+a_2-j_1$, and when $j_2>a_2$, this is $\bar{j_1}$. Altogether, we get $U-L=\min(\bar{j_2},a_2,\bar{j_1})$.
					
				\end{enumerate}
				
			\end{enumerate}
			
		\end{proof}
		\begin{remark}
			We conjecture that the $e$-regular multipartitions in a face with $I_0=\{1,2,\dots,t\}$ can be obtained from a path with residues $t,t-1,\dots,1,2,\dots,t,t-1,\dots,3,2,3,\dots,t-2,t-1,t-2,\dots$.
			The first sweep from $t$ to $1$ fills in
			the first column of each partition and the return from $1$ to $t$ fills in the first row. The $1$-corner  and $t$-corner partitions are now finished, and we continue to the second row and column for the partitions with corner in the interval $2,3,\dots, t-2$, continuing to remove the outer two residues at each cycle.  
		\end{remark}
		
		\section{THE SHAPE OF CANONICAL BASIS ELEMENTS}

		\begin{defn}
			The \textit{shape} of a canonical basis element $G(\mu)$ is the sequence of coefficients in the generating function $f_\mu(z)$ obtained from $G(\mu)$ by substituting a $1$ for each element of the natural basis and $z$ for each $v$. Since the powers of $v$ occurring in $G(\mu)$ are 
			non-negative, $f_\mu(z) \in \mathbb{Z}[v]$ is an integral polynomial with non-negative coefficients.
		\end{defn}
		
		In many cases, the shape is much more quickly and easily calculated than the massive canonical basis elements. For any weight,  the dimension of the weight space is at least as large as the number of distinct shapes.
		
		For each vertex of $\widehat P(\Lambda)$ of the form $\Lambda -j_i\a_i$, there is a unique path leading to the vertex, and thus the dimension of the corresponding weight space in $V(\Lambda)$ is one. We want to calculate the shape of any such canonical basis element.

		\subsection{Monotonic paths}

		We now give a recursion formula for the shape and solve it by giving the generating function in terms of known functions. We first need a few definitions. 
		
		A \textit{binary word} is a word such that each letter belongs to $\{0,1\}$. A binary word of length $a$ consisting entirely of $0$ will be denoted by $0_a$.
		
		An {\em inversion} in a binary word  $\pi$ is a pair of positions $(i,j)$ such that $1=\pi_i>\pi_j=0$ and $1\leq i<j\leq n$. 
		\begin{defn}\label{bin}	
			We let $\Inv(S)$ be the number of inversions in the binary word $S$. If $T, S$ are binary words of the same length such that the set of positions of digits $1$ in $T$ is a subset of the positions of digits $1$ in $S$, we say that  $T \subseteq S$.  In such a case, 
			we let $\Inv(T, S)$ be the relative inversion of the values $1$ of $T$ as a subset of the values $1$ of $S$, that is to say, the sum of the number of $1$ digits in $S-T$ before each digit $1$ in $T$.
			If $X \subset U \subset S$, then we exclude the digits in the word $S$ corresponding to $X$, and for each $1$ in $S-U$ take the sum for each position before it of 
			$1$ for each $0$ digit and $-1$ for each $1$ from $U$.
			The sum for all digits $1$ in $S-U$ will be denoted by  
			$\Inv(S,U,X)$. We let $\#1(S)$ denote the number of digits $1$ in the binary word $S$, so that $\#1(0_{a})=0$.
			
			For nonnegative integers $a,j$ 
			we let $B(a,j)$ denote the set of binary words of length $a$ with $j$ entries $1$ if $0 \leq j \leq a$ and let it be $\emptyset$ otherwise.
		\end{defn}
		
		In all the results below, we will frequently use Theorem 6.16 in Mathas \cite{M2}. The result there is stated for partitions rather than multipartitions, and Mathas is working with $e$-restricted rather than $e$-regular multipartitions, so in \cite{AOS1}, to take care of these differences, we gave a slightly different proof compatible with our set-up.  However,  except when the differences are important, we will simply quote Mathas.
		
		\begin{lem}\label{Mathas}[\cite{AOS2}, Lem. 4.2] If a multipartition $\mu$ has at least $t$ $i$-addable nodes, and if $\lambda$ is the result of adding $t$ $i$-addable nodes with binary word $S$ choosing among addable nodes, then $f^{t}_i(\mu)$ contains $\lambda$ with coefficient $v^{\Inv(S)}[t]_v!$.
		\end{lem}

		\begin{defn}
			Define $(z;z)_j$ to be $\prod_{t=1}^j(1-z^t)$ with $(z;z)_0=1$, and $\binom{i}{j}_z=\frac{(z;z)_i}{(z;z)_j(z;z)_{i-j}}$ to be the Gaussian binomial coefficients.
		\end{defn}

		\begin{prop}\label{shape} Let $\Lambda=a_0 \Lambda_0+\dots+a_\ell \Lambda_{\ell}$. For some $i$ in the residue set $I$, suppose we are given 	
			$j_i$ with $0 \leq j_i \leq a_i$. Let $\mu$ be the unique $e$-regular
			multipartition of weight $\Lambda-j_i\alpha_i$. We set $a=a_i$, $j=j_i$.
			\begin{enumerate}
				\item The shape of the canonical basis element
				$G(\mu)$ is given by $(s(a,j,0),\dots, s(a,j,j(a-j))$ where $s(a,j,t)$ is a recursive function which is $0$ except for $0 \leq t \leq j(a-j)$, and satisfies the following recursion scheme:
				\[
				s(a,j,t)=s(a-1,j-1,t)+s(a-1,j,t-j).
				\]
				\noindent with initial conditions s(1,0,0)=s(1,1,0)=1.
				\item The generating function $f_\mu$ given by the above recursion scheme is, for all $0\leq j\leq a$,
				\[S(a,j)(z)=\sum_{t=0}^{j(a-j)}s(a,j,t)z^t=\binom{a}{j}_z,
				\]
				where $\binom{a}{j}_z$ is the Gaussian binomial coefficients. Thus, for all $0\leq j\leq a$ and $0\leq t\leq j(a-j)$, $s(a,j,t)$ is the number of binary words $\pi=\pi_1\pi_2\cdots\pi_j$ with $j-a$ digits $1$, $a$ digits $0$ and having $t$ inversions.
			\end{enumerate}
		\end{prop}
		
		\begin{proof}
			(1)  Let $\mu$ be the unique $e$-regular multipartition with weight $\Lambda-j_i\alpha_i$.
			The function $s(a_i,j_i,t)$ which gives the shape is a function counting all the multipartitions with coefficient $v^t$ in $G(\mu)$.
			
			For $a=1$, we have  $j=0,1$.  This means that the defect $j(a-j)=0$, so $t=0$.   For a  canonical basis element with defect $0$, $G(\mu)= |\mu \rangle$, so  $s(1,1,0)=s(1,1,0)=1$.  
			
			For $a>1$, any multipartition in $G(\mu)$ has singletons $(1)$ whose $j$ positions are determined by a binary word $S$.
			For the recursion step from $a-1$ to $a$, the  
			number of multipartitions with coefficient $v^t$ in $G(\mu)$ is the sum of those starting with $1$, for which the power is determined by the remaining $a-1$ elements of the sequence, and those starting with $0$. Let $S$ be a binary word starting with $0$ determining one of the multipartitions. Letting $S'$ be the characteristic sequence when we have removed the initial $0$, we know that $S'$ also contains $j$ values $1$.  We have $\Inv(S) =\Inv(S')+j$  because the initial $0$ adds $j$ to the power of $v$  determined by the remainder of the sequence, giving the desired recursion formula.
			
			(2)	 Note that $s(a,j,t)$ is defined for all $0\leq j\leq a$ and $0\leq t\leq j(a-j)$. (Otherwise $s(a,j,t)$ is defined to be $0$).
			
			To solve this recurrence relation, we define $S(a,j)(z)$ to be the generating function $\sum_{t=0}^{j(a-)j}s(a,j,t)z^t$. By multiplying the recurrence relation of $s$ by $z^t$ and summing over all possible values of $t$, we obtain
			\begin{align}\label{eqR1}
				S(a,j)(z)=S(a-1,j-1)(z)+z^jS(a-1,j)(z),
			\end{align}
			with $S(1,0)(z)=S(1,1)(z)=1$. 
			
			Now, we show that $S(a,k)(z)=\binom{a}{j}_z$. We proceed with the proof by induction on $a+j$. Clearly, $S(1,0)(z)=\binom{1}{0}_z$ and $S(1,1)(z)=\binom{1}{1}_z$. Assume that the claim holds for all $a,j$ such that $a+j\leq s-1$. By \eqref{eqR1} and induction hypothesis, we have 
			\begin{align*}
				S(a,j)(z)&=S(a-1,j-1)(z)+z^jS(a-1,j)(z)=\binom{a-1}{j-1}_z+z^j\binom{a-1}{j}_z\\
				&=\frac{(z;z)_{a-1}}{(z;z)_{a-j}(z;z)_{j-1}}+\frac{z^j(z;z)_{a-1}}{(z;z)_{a-1-j}(z;z)_{j}}\\
				&=\frac{(z;z)_{a-1}}{(z;z)_{a-j}(z;z)_{j}}(1-z^j+z^j(1-z^{a-j}))\\
				&=\frac{(z;z)_{a}}{(z;z)_{a-j}(z;z)_{j}}=\binom{a}{j}_z,
			\end{align*}
			which ends the induction. The rest follows immediately from the definitions of the Gaussian binomial coefficients (for instance, see \cite{EE}).
		\end{proof}
		
		Note that since the Gaussian binomial coefficients are symmetric, we see that the polynomial $S(a,j)(z)=\sum_{t=0}^{j(a-j)}s(a,j,t)z^t$ is symmetric too, namely $S(a,j)(z)=S(a, a-j)(z)$, for all $0\leq j\leq a$.
		Moreover, By Prop. \ref{shape}, we see that the shape of the canonical basis element $G(\mu)$ is given by the coefficient of the Gaussian binomial coefficients $\binom{a_i}{j_i}_z$.
		
		\subsection{Paths $i^{j_i}k^{j_k}$}
		
		We now consider the shape of the canonical basis element reached by a path consisting of operating $j_i$ times by $f_i$ with $j_i \leq a_i$, and $j_k$ times by $f_k$ with $j_k \leq a_k$. We first consider the case where $k$ is not congruent to $i+1$ or $i-1$ modulo $e$. Cases with $i,k$ adjacent will be special cases of the paths that will be treated in \S \ref{three}.
		\begin{defn} 
			The \textit{support} of a multipartition is the set of residues in nodes of non-empty partitions.  If $\mu$ and $\nu$ are two multipartitions for the same multicharge $s$ which have disjoint support, then we define $\mu \wedge \nu$ to be the combined multipartition equal to $\mu$ on the support of $\mu$ and equal to $\nu$ on the support of $\nu$. If two canonical basis elements have disjoint support, we can define $G(\mu) \wedge G(\nu)$ by linearity.
		\end{defn}
		
		\begin{lem}	If $i$ and $j$ are non-adjacent elements of the index set $I$, then
			\begin{enumerate}
				\item If $\mu_i$ is the $e$-regular multipartition for the path $i^{j_i}$ and
				and $\mu_k$ for the path $k^{j_k}$, then for the vertex with content $c_i=j_i$ and $c_j=j_k$ but otherwise zero, there is a unique canonical basis element 
				$G(\mu_i) \wedge G(\mu_k)$ and  the shape of the element has generating polynomial $f_{\mu_i}(z)f_{\mu_k}(z)$.
				\item There is only one e-regular multipartition $\mu$ with the weight of $\Lambda-j_i\Lambda_i-j_k\Lambda_k$, given by $\mu_i \wedge \mu_j$.	
			\end{enumerate}	
		\end{lem}
		
		\begin{proof}
			\begin{enumerate}
				\item Induction on $j_i$.  For $j_i=0$, we have settled this case above, that the multipartition $\mu_k$ for any given $j_k$ has support only on the $k$-corner partitions and $G(\mu_k)$ is supported entirely on these partitions.  When we act on $G(\mu_k)$ by the operator $f_i$, each $v^t\vert \nu \rangle$ is changed only in its $i$-corner partitions.  For each possible binary word $S$, let $\mu_S$ be the result of taking the partition designated by the single $1$ in $S$.  For each $S$, we will then get $v^{\Inv(S)}v^t |\mu_S \wedge \nu \rangle$.  By the distributive law, this is equivalent to $(\sum_{S} v^{\Inv(S)})  |\mu_S \wedge \nu>$. Because the operation of $f_i$ and $f_k$ are independent, along any path leading to the given vertex, 		
				we will have $G(\mu)=G(\mu_i)G(\mu_k)$.  Then $f_{\mu}(z)$ is obtained by substituting $1$ for all the multipartitions in $G(\mu)$ and replacing $v$ by $z$.
				\item The $e$-regular multipartition with the weight of  $\Lambda-j_i\alpha_i$ consists, by the signature method, of $j_i$ copies of $(1)$ in the $i$-corner partitions and this will give the multipartition $\mu_i$. Similarly the $e$-regular partition $\mu_k$ with weight  $\Lambda-j_k \alpha_k$ will have $j_k$ copies of $(1)$ in the $k$-cornered partitions. There are no $k$-addable nodes in the $i$-corner partitions because the only addable nodes to $(1)$ have residue $i+1$ in the row and $i-1$ in the column. Similarly, there are no $i$-addable nodes in the $k$-corner partitions. Thus, whatever path we choose, if we arrive at a content with only $c_i$ and $c_k$ non-zero, the $e$-regular partitions from that path will have $c_i$ copies of $(1)$ in $i$-corner partitions and $c_k$ copies of $(1)$ in $k$-corner partitions.    
			\end{enumerate}	
		\end{proof}
		
		\subsection{Paths $k^{u}i^{j_i}k^{j_2-u}$for $i,k$ adjacent}\label{three} 
		To make our notation compatible with Prop. \ref{e-reg}, we will take $i=1$ and $k=2$. This can be arranged through a cyclic renumbering of the residues.
		In Prop. \ref{e-reg}, we showed that the $e$-regular multipartitions in the upper part of the face are 
		reached by paths of the form $2^{u}1^{j_1}2^{j_1-u}$, where $u=\min(j_1-w,j_2-w)$. For paths with  $2^u$ or $2^u1^{j_1}$, we will find the shape for the Fock space element associated with each path. 
		For the full path, we must find the Fock space element and get the shape from there, but first, we need a notation for the multipartitions. We define the following binary words as in Def \ref{bin}.
		\begin{defn}\label{TWSX} Given $j_1,j_2$ with $0 \leq j_1,j_2 \leq a_1+a_2$ and given $u$ with $0 \leq u \leq \min(a_2, j_2)$, we choose binary words $T$, $W$, $S$, $X$ as follows.\\
			\begin{itemize}
				\item We chose a binary word $T \in B(a_1,t)$, where $t$ is an integer
				with $\max(0,j_1-u) \leq t \leq 
				\min(a_1,j_1)$ 
				.
				\item   We choose a word 
				$X \in B(a_2,j_1-t)$. 
				\item We choose a binary word
				$W \in B(a_1,w)$ , $W \subseteq T$, where   $w$ is an integer with $\max(0,j_2-a_2) \leq w \leq \min(t,j_2-u)$, 
				
				\item We choose a binary word $S \in  B(a_2,j_2-w)$, with $X \subseteq S$. 
			\end{itemize}
			The multipartition 
			\[
			\tau(T,W,S,X)
			\] is the multipartition with partitions 
			$(2)$ in positions given by $W$, $(1)$ in positions given by $T-W$, $(1,1)$ in positions given by $X$, and $(1)$ in positions given by $S-X$, with all remaining partitions equal to $\emptyset$.
		\end{defn}
		
		\begin{prop} For $e \geq 3$ and $\Lambda=a_0\Lambda_0+\dots+a_\ell\Lambda_{\ell}$, on the face for  interval  $I_0=\{1,2\}$  at the vertex with content $(j_1,j_2)$, we set $t=min(j_1,a_1)$,choose an integer $w$ satisfying
			$\max(0,j_2-a_2) \leq w \leq \min(t,j_2, a_1+j_2-j_1) $ and	
			denote by $\mu_w$ the $e$-regular multipartition with $w$ copies of the partition $(2)$. Set $u=\min(j_1-w,j_2-w)$.
			Consider the Fock space element 
			$F(j_1,j_2,u)=f_2^{(j_2-u)}f_1^{(j_1)}f_2^{(u)}(\lvert ~ \rangle) $ at the endpoint of the  path $2^{u}1^{j_1}2^{j_2-u}$.
			\begin{enumerate}
				\item At the end of the first subpath  $2^{u}$, the generating function of the shape is $f(z)=\binom{a_2}{u}_z$.
				\item At the end of the subpath  $2^{u}1^{j_1}$ the generating function of the shape is f(z)g(z), where $g(z)=\binom{a_1+u}{j_1}_z$. If $u=j_2$, this is the end of the path and thus the shape is the product of two Gaussian binomials. If $u=0$, then the shape is $g(z)$.
				\item \label{3} In the notation of  Def. \ref{TWSX},
				\[F(j_1,j_2,u)=\sum_{T,W,S,X}
				v^{E(T,W)}
				\sum_{X\subset U\subset S}v^{E(S,U,X)}\lvert \tau(T,W,S,X)>,
				\]
				where 
				\begin{itemize}
					\item $E_X(t)=(a_1-t)(j_1-t),$
					\item $E_S(t,w)=(t-w)(j_2-u-w),$
					\item $E(T,W)= \Inv(T)+\Inv(W,T)+E_X(t)+E_S(t,w)$
					\item $E(S,U,X)={\Inv(U)}+\Inv(X,U)+\Inv(S,U,X).$
					
				\end{itemize}
				\item The Fock space element $F(j_1,j_2,u)$ is a linear combination of canonical basis elements $G(\mu_{w_1})$ for $w_1 \geq w$ with coefficents in $\mathbb{Z}[v+v^{-1}]$. If we calculate canonical basis elements starting with maximal $w$, and for each new $w$ we strip off those canonical basis elements with 
				$w_1 > w$, we are left with the canonical basis element element $G(\mu_w)$.
				
			\end{enumerate} 
			
		\end{prop}	
		\begin{proof}
			\begin{enumerate}
				\item 	We start the path choosing $u \leq a_2$ from among the  $2$-corner partitions, of which there are $a_2$, in which to insert $(1)$, with the generating function $f(z)$ given in Prop. \ref{shape}.
				For each such choice, we denote by $U \in B(a_2,u)$ the corresponding binary word. The coefficient of the corresponding multipartition $\tau(0_{a_1},0_{a_1},U,0_{a_2})$ is $v^{\Inv(U)}$.
				\item From the bound $w \leq a_1+j_2-j_1$, we have 
				$j_1-a_1 \leq j_2-w$. Since $w \leq a_1$, we also have $j_1-a_1 \leq j_1-w$. We chose $u = \min(j_1-w,j_2-w)$, so that $j_1-a_1 \leq u$ and thus $j_1 \leq a_1+u$.
				Now we have $a_1+u$ addable $1$-nodes, of which we will choose $j_1$.  
				Of these, some will add $1$-corner partitions, and some will be bottom nodes of partitions of the form $(1,1)$. The precise locations of the $2$-corner partitions $(1)$ do not affect the choices, so the generating function of the shape would be the product of two Gaussian binomials, the first being the polynomial $f(z)$ in (1) of the proposition and the second being of the Gaussian binomial $g(z)=\binom{a_1+u}{j_1}_z$.
				If $u=0$, then $f(z)=1$, so this product is just $g(z)$.
				
				\item We defined $u=\min(j_1-w,j_2-w) \leq j_2-w$ and $w \geq 0$, so $u \leq j_2$. We construct the Fock space element $F(j_1,j_2,u)$  by induction on $y=j_2-u$, starting with $y=0$.  We first calculate the coefficient of $\tau'=\tau(T,0_{a_1}, U, X)$, a typical multipartition whose corresponding natural basis element $\lvert \tau' \rangle$ appears in $f_2^{(j_1)} \left (f_1^{(u)}\right )(\lvert ~ \rangle) $. We have chosen a binary word  $T \in B(a_1,t)$ which indicates the positions of the 
				$1$-corner partitions of form $(1)$, and 
				$X \in B(a_2,j_1-t)$, with $X \subseteq U$ is the binary word representing the locations of the $2$-corner partitions $(1,1)$. The numbers $t$ and $j_1-t$ are nonnegative, and in addition $t\leq a_1$ while $j_1-t \leq u$. This gives the bounds for $t$ which we quoted in Def. \ref{TWSX}.
				
				We let $\widehat{X}$ be a binary word of length $u$ which is $0$ except at the relative positions of the $1$-values of $X$ among the $1$ values of $U$. By Lemma \ref{Mathas}, the coefficient of $\tau'$ is the number of inversions in the join of $T$ and $\hat{X}$, which we now calculate.
				
				The number of inversions is the sum of the number of $0$ digits to the left of any $1$, so in counting the number of inversions for the digits in $T$, the digits of $\widehat{X}$ are irrelevant and we just get $\Inv(T)$.  For each digit $1$ in 
				$\widehat{X}$, we add together the $a_1-t$ digits $0$ in $T$, together with the $0$ digits proceeding that $1$ in $\widehat X$.  Since there are $j_1-t$ digits $1$ in $\widehat X$, we get altogether 
				$(a_1-t)(j_1-t)+\Inv(\hat{X})$.  We have denoted 	$(a_1-t)(j_1-t)$ by $E_X(t)$ in the statement of the proposition and we have $\Inv(\widehat{X})=\Inv({X},U)$ by Def. \ref{bin}. When we add in the coefficient $v^{\Inv(U)}$ from the placement of $U$ in the $2$-corner partitions, we get a total exponent $\Inv(U)+\Inv(T)+E_X(t)+\Inv(X,U)$.
				
				To shorten the formula, let us define
				\[
				E(T,W,S,U,X)=E(T,W)+{E(S,U,X)}.
				\]	
				If $j_2=u$, then $W=0_{a_1}$ so $\Inv(W,T)=0$ and $S=U$, so that $\Inv(S,U,X)=0$ and $E_S(t,0) =0$.
				\begin{align*}
					E(T,W,S,S,X)=&E(T,W)+ {E(S,S,X)},\\
					=& \Inv(T)+\Inv(W,T)+E_X(t)+E_S(t,w)\\
					+&{\Inv(S)}+\Inv(X,S)+\Inv(S,S, X),\\
					=&\Inv(T)+E_X(t)+ {\Inv(S)}+\Inv(X,S),\\ 
				\end{align*}
				
				This is precisely the exponent we get when we substitute $S$ for $U$ in the total exponent  $\Inv(U)+\Inv(T)+E_X(t)+\Inv(X, U)$ calculated above. Thus we have established (3) when $j_2-u=0$.  
				
				We now proceed by induction on 
				$y=j_2-u$. Assume we have proven (\ref{3}) for $y-1 \geq 0$ and want to prove it for $y$. If $w \neq 0$, we let $W_1, W_2, \dots W_w$ be the binary words obtained by successively turning one of the digits $1$ of $W$ to $0$. For each $U$ with  $X\subset U\subset S$, letting $r$ be the number of digits $1$ in $S-U$, we let $S_1^U, S_2^U, \dots S_r^U$ be the binary words obtained by setting one of the digits $1$ in $S-U$ to $0$. The integer $r$ is independent of $U$ and together $w+r=y$. We now act on 
				
				\[\sum_{X\subset U\subset S}\sum_{i=1}^w v^{E(T, W_i,S,U,X)}\lvert \tau(T,W_i,S,X)>
				+ \sum_{j=1}^r
				v^{E(T,W,S_j^U,U,X)}\lvert \tau(T,W,S_j^U,X)>
				\]
				by the operator $f_2$ and isolate the term with multipartition
				$\tau=\lvert \tau(T,W,S,X)>$.  Since we have acted $y$ times by $f_2$ and have divided by the factorial $[y-1]_v!$, the coefficient should be divisible by $[y]_v$. 
				
				For the calculation of the first part with the binary word $W_i$, all factors of $E(T, W_i, S, U, X)$ are fixed and identical to those of $E(T, W, S, U, X)$ except that we have $E_S(t, w-1)$ instead of $E_S(t,w)$, $j_2'=j_2-1$ and we have a variation in  $E(W_i, T)$. With $S$ fixed, 
				\begin{align*}
					E_S(t,w-1)-E_S(t,w)=&
					-(t-(w-1))(j_2-1-u-(w-1))\\&-(t-w)(j_2-u-w)\\
					=&j_2-u-w\\
					=&y-w=r
				\end{align*}
				
				For each $i$, there are $i-1$ digits $1$ of $W$ before the vacant position in $W_i$, and some number $N_i$ of digits $0$ from $T-W$, with
				$\Inv(W,T)=\sum_{i=1}^w N_i$. When we calculate $\Inv(W_i,T)$, each of the
				$w-i$ integers $j>i$ contributes $N_j+1$ to $\Inv(W_i,T)$ because the empty spot in position 
				$i$ of $W$ gives an extra $0$. Altogether
				\[
				\Inv(W_i,T)=\Inv(W,t)-N_i+w-i.
				\] 
				
				Thus $\Inv(W_i,T)+N_i=\Inv(W,T)+w-i$.  When we operate on $\tau(T,W_i,S,X)$ by $f_2$, the $i$-dependent part of the coefficient of $\tau$
				will be $\Inv(W_i,T)+E_S(t, w-1)$, to which we add $N_i$ for the addable $2$-nodes and subtract $(i-1)$ for the removable $2$-nodes.  Summing over $i$ we get a $W$-dependent coefficient $C_W$. 
				\begin{align*}
					C_W = &\sum_{i=1}^w  v^{\Inv(W_i,T)+N_i-(i-1)+E_S(t, w-1)}\\
					=&\sum_{i=1}^wv^{\Inv(W,T)+(w-i)-(i-1)+(E_S(t,w)+r)}\\
					=&v^{\Inv(W,T)+E_S(t,w)}v^{r}\sum_{i=1}^wv^{(w+1-2i)}\\
					=&v^{\Inv(W,T)+E_S(t,w)}v^{r}[w]_v\\ 	 
				\end{align*}
				
				Now we fix $U$ and calculate the part of the coefficient which varies as we omit 
				one of the $r$ digits $1$ located in the positions of $S-U$. Since $U$ is fixed, we shorten $S^U_j$ to $S_j$. The variation is in $E_{S_j}(t,w)$ and in $\Inv(S_j,U,X)$. We will need
				
				\begin{align*}
					E_S(t,w)-E_{S_j}(t,w)=&
					(t-w)(j_2-u-w)-(t-w)(j_2-u-1-w)\\
					=&t-w\\
				\end{align*}
				
				For each $j$, let $N_j^+$ be the number of addable nodes before the digit $1$in position $j$ of $S-U$ and let $N_j^-$ be the number of digits $1$ of $U-X$ before position $j$.
				By definition, $\Inv(S,U,X)= \sum_{j=1}^r (N_j^+-N_j^-)$.  In $\Inv(S_j,U,X)$, we are not only missing the term $N_j^+-N_j^-$, but if we let $\bar N_i^+,\bar N_i^-$ be the corresponding numbers used to calculate $\Inv(S_j,U,X)$ for $i \neq j$, we also have $\bar N_i^+=\bar N_i^++1$ for $i<j$ and $\bar N_i^-=N_i^--1$ for $i>j$. Thus
				\[
				\Inv(S,U,X)=\Inv(S_j,U,X)+N_j^+-N_j^-+j-1-(r-j).
				\] 
				When we act by $f_2$, we count up all the addable nodes before position $j$, which is $N_j^++(t-w)$, and subtract the number of removable nodes, which is $N_j^-+w$. Thus the $S$ variable part of the coefficient of 
				$\tau(T,W,S_j,X)$ is multiplied by $v^{N_j^+-N_j^-+t-2w}$. Summing over $j$,
				we get that the factor in the coefficient of $\tau$ depending on $S$ is 
				\begin{align*}
					C_S=&\sum_{j=1}^r v^{\Inv(Sj,U,X)+N_j^+-N_j^-+(t-2w)+E_{S_j}(t,w)}\\
					=&\sum_{j=1}^r v^{\Inv(S,U,X)-r+2j-1+(t-2w)+(E_{S}(t,w)-(t-w))}\\
					=& v^{\Inv(S,U,X)+E_S(t,w)-w}\sum_{j=1}^r v^{-r+2j-1}\\
					=& v^{\Inv(S,U,X)+E_S(t,w)}v^{-w}[r]_v
				\end{align*}
				
				Still keeping $U$ fixed, we add together the coefficient of $\tau$ coming from the variation of $W$ and $S$.
				\begin{align*}
					C=&v^{\Inv(U)+\Inv(T)+\Inv(X,T)+E_X(t) }
					\left [v^{\Inv(S,U,X)}C_W+v^{\Inv(W,T) }C_S\right ]\\
					=&v^{\Inv(U)+\Inv(T)+\Inv(X,T)+E_X(t)+E_S(t,w)+\Inv(S,U,X)+\Inv(W,T) }
					\left [v^{r}[w]_v+v^{-w }[r]_v\right ]\\
					=&v^{\Inv(U)+\Inv(T)+\Inv(X,T)+E_X(t)+E_S(t,w)+\Inv(S,U,X)+\Inv(W,T) }
					[y]_v 
					\\
				\end{align*}
				
				When we divide out by $[y]_v$ to get the divided power, we have the desired exponent.
				Finally, we take the sum over all $U$ to get the  result in the proposition:

				\[\sum_{T,W,S,X}
				v^{E(T,W)}
				\sum_{X\subset U\subset S}v^{E(S,U,X)}\lvert \tau(T,W,S,X)>
				\]
				where 
				\begin{itemize}
					\item $E_X(t)=(a_1-t)(j_1-t),$
					\item $E_S(t,w)=(t-w)(j_2-u-w),$
					\item $E(T,W)= \Inv(T)+\Inv(W,T)+E_X(t)+E_S(t,w)$,
					\item $E(S,U,X)={\Inv(U)}+\Inv(X,U)+\Inv(S,U,X).$
					
				\end{itemize}
				
				For a compacter representation of the formula, we may define
				\[
				a(T,W,,S,X)=v^{E(T,W)}\sum_{X\subset U\subset S}v^{E(S,U,X)},
				\]
				and write 
				
				\[\sum_{T,W,S,X}
				a(T,W,S,X)\lvert \tau(T,W,S,X)>.
				\]
				\item In Prop. \ref{e-reg}, we determined that all the 
				$e$-regular multipartitions for $(j_1,j_2)$ have 
				$w$ satisfying $\max(0,j_2-a_2) \leq w \leq \min(a_1, j_1,j_2, a_1+j_2-j_1)$. We are given a specific value of $u$, which corresponds to some $w$ in the given range.

				Now choose a value $w_0<w$ in this range, if such exists, and let $u_0=\min(j_1-w_0, j_2-w_0)$, so that we get 
				$u_0 > u$.  We want to show that the coefficient of $\lvert \mu_{w_0} \rangle$ is a polynomial in $v$ with positive exponents, and so cannot be a polynomial in $v+v^{-1}$.	The $e$-regular multipartitions appearing in the path $2^{u}1^{j_1}2^{j_2-u}$ all have   $t=\min(a_1,j_1)$, $x=\max(0, j_1-a_1)$
				for any of the $w$ in that range, so $t,x$ are the same for $w_0$ and $w$, as are $T$ and $X$. However, we do have different $W_0$ and $S_0$.  
				We calculate the exponent of $\mu_{w_0}$ using the 
				formula in (\ref{3}).
				\[
				v^{E(T,W_0)}
				\sum_{X\subset U\subset S_0}v^{E(S_0,U,X)}\lvert \tau(T,W_0,S_0,X)>.
				\] 
				By the rules for choosing $i$-cogood nodes, we have 
				$\Inv(T)=\Inv(W_0,T)=0$ for an $e$-regular multipartition because we add nodes, first $1$-nodes to form $T$ and then $2$-nodes to form $w$,  from the beginning. We also have  $E_X(t)=(a_1-t)(j_1-t)=0$ because $t$ equals either $a_1$ or $j_1$. Thus 
				$E(T,W_0)=E_S(t,w_0)=(t-w_0)(j_2-u-w_0)\geq 0$.
				Now $\#1(S_0)=j_2-w_0>j_2-w \geq u$, so $j_2-u-w_0>0$. Furthermore, $t \geq w >w_0$, so 
				$t-w_0$ is also positive.  Thus $E_S(t)>0$.
				
				We now consider $U$ with $\#1(U)=u_0$ such that 
				$X \subseteq U \subseteq S_0$. We have 
				$E(S_0,U,X)={\Inv(U)}+\Inv(X,U)+\Inv(S_0,U,X)$.
				Because $\mu_{w_0}$ is $e$-regular, all the partitions $(1,1)$ are at the beginning of $U$, so $\Inv(x,U)=0$. 
				
				If we can show that 
				$E_{S_0}(t)+E(S_0,U,X)$ is always positive, we have shown that 
				$\lvert \tau(T, W_0, S_0, X)>$ does not occur as the leading element in a canonical basis elements multiplied by a polynomial in $v+v^{-1}$.
				The number $\Inv(S, U, X)$ can be negative, but it cannot be less than $-(u-x)(j_2-u-w_0)$, which minimum occurs when all the $u-x$ digits $1$of $U-X$, each adding $-1$ to $\Inv(S, U, X)$ for each digit $1$ of $S_0-U$, are before all the $j_2-u-w_0$  digits $1$ of
				$S_0-U$.  Thus
				
				\[
				E_{S_0}(t)+E(S_0,U,X) \geq (t-w_0)(j_2-u-w_0)-(u-x)(j_2-u-w_0)
				\]
				\[		
				=[(t-w_0)-(u-x)](j_2-u-w_0).		
				\]
				We showed above that $j_2-u-w_0>0$, so we need only show that $t-w_0>u-x$.
				Then $u=\min(j_1-w,j_2-w)\leq j_1-w$.  Substituting 
				$j_1=t+x$, we have $u-x \leq t-w < t-w_0$. 
				Thus 
				$E_{S_0}(t)+E(S_0,U,X) >0$, as desired.
				
				The same arguments show that for $w$, 
				the binary word $U$ which has all the 
				digits $1$ at the beginning has
				$E_{S}(t)+E(S,U,X)=0$.  By the choice of $U$, we have $\Inv(U)=0$.
				We still have $\Inv(X,U)=0$. Finally,
				$E(S,U,X)=(-1)(u-x)(j_2-u-w)$, while
				$E_S(t)=(t-w)(j_2-u-w)$. There are two
				cases. If $j_2<j_1$, then $u=j_2-w$,
				so both $E(S,U,X)$ and $E_S(t)$ are zero. If $j_2 \geq j_1$, then 
				$u=j_1-w$, and substituting 
				$j_1=t+x$ gives $u-x=t-w$, so
				$E_{S}(t)+E(S,U,X)=0$ and the coefficient of $\lvert \mu_w \rangle$ for this binary word $U$ is $1$.  
				
			\end{enumerate} 
		\end{proof}

		\begin{examp}
			For $e=4$, we calculate the canonical basis elements for a path $2^11^22^2$ with $a_1=2, a_2=3$, so that $j_1=2,j_2=3,w_1=1$. The output, from the program fock\_space, uses ``q'' for ``v''.
			\begin{itemize}
				\item $G([[], [], [1], [], []])=|[], [], [1], [], []> + q*|[], [], [], [1], []> + q^2*|[], [], [], [], [1]>$
				
				$f(z)=1+z+z^2$.
				\item $G([[1], [1], [1], [], []] )=
				\left(|[1], [1], [1], [], []> + q*|[1], [1], [], [1], []> + q^2*|[1], [1], [], [], [1]>\right) + q*|[1], [], [1, 1], [], []> + q^2*|[1], [], [], [1, 1], []> + q^3*|[1], [], [], [], [1, 1]> + q^2*|[], [1], [1, 1], [], []> + q^3*|[], [1], [], [1, 1], []> + q^4*|[], [1], [], [], [1, 1]>$.
				
				\noindent The generating function of the shape is 
				$f(z)g(z)=(1+z+z^2)^2= 1+2z+3z^2+2z^3+z^4$
				with $g(z)=g_0(z)+g_1(z)$,
				where
				$g_0(z)=1$ and $g_1(z)=z \cdot \binom{2}{1}=z+z^2$.
				\item $ G([[2], [1], [1], [1], []])=(|[2], [1], [1], [1], []> + 
				q*|[2], [1], [1], [], [1]> + q*|[1], [2], [1], [1], []> + 
				q^2*|[2], [1], [], [1], [1]> +  
				q^2*|[1], [2], [1], [], [1]> 
				q^3*|[1], [2], [], [1], [1]>) + 
				(q^4+q^2)*|[1], [1], [1], [1], [1]> + 
				q^2*|[2], [], [1, 1], [], [1]> + q^2*|[2], [], [1], [1, 1], []> +
				(q*|[2], [], [1, 1], [1], []> + q^2*|[], [2], [1, 1], [1], []> +
				q^3*|[2], [], [1], [], [1, 1]> +q^3*|[], [2], [1, 1], [], [1]> + q^3*|[], [2], [1], [1, 1], []> + q^3*|[2], [], [], [1, 1], [1]>  +
				q^4*|[2], [], [], [1], [1, 1]> q^4*|[], [2], [], [1, 1], [1]> +q^4*|[], [2], [1], [], [1, 1]> + q^5*|[], [2], [], [1], [1, 1]>) +
				( q^3*|[1], [], [1, 1], [1], [1]> +
				q^4*|[], [1], [1, 1], [1], [1]> +q^4*|[1], [], [1], [1, 1], [1]> +
				q^5*|[1], [], [1], [1], [1, 1]> +q^5*|[], [1], [1], [1, 1], [1]> + 
				q^6*|[], [1], [1], [1], [1, 1]>)$.

				For a choice of $t,w$, let $h_{t,w}$ be the corresponding summand of the generating function.  Then
				
				$h_{2,1}=(1+z)(1+z+z^2)=1+2z+2z^2+z^3$,
				
				$h_{2,0}=z^2+z^4$,
				
				$h_{1,1}=(1+z)^2z(1+z+z^2)=z+3z^2+4z^3+3z^4+z^5$,
				
				$h_{1,0}=(1+z)z^3(1+z+z^2)=z^3+2z^4+2z^5+z^6$.
				
				The polynomials $h_{2,0}$ and $h_{1,1}$ are self-dual with respect to the midpoint $z^3$ while $h_{2,1}$ and $h_{1,0}$ are dual to each other. Adding them all together we have a symmetric generating function.
				
				$1+3z+6z^2+6z^3+       6z^4 +3z^5+z^6$.		
			\end{itemize}	
		\end{examp}
		
		\begin{examp}
			We give an example to illustrate the case where a path produces a Fock space element which is a non-integer combination of canonical basis elements, with $a_1=3$, $a_2=5$. The path is $2^31^42^3$, of defect $14$, and produces a Fock space element $M$ with
			\[
			M=G([[2],[2],[2],[1,1][1],[1],[],[]])+(v+v^{-1})G([[2],[2],[1],[1,1][1],[1],[1],[]])
			\]
			\[
			+G([[2],[1],[1],[1,1][1],[1],[1],[1]]),
			\]
			\noindent where the last of the three canonical basis elements $G(\mu_1)$ remains when the first two, $G(\mu_3)$ and $(v+v^{-1})G(\mu_2)$, are stripped off.  The first can be reached by a path with $u=1$ and the second by a path with $u=2$.
		\end{examp}
		\section{OPEN PROBLEMS}
		This treatment leaves several open procblems:
		\begin{enumerate}
			\item If we project $\max(\Lambda)$ onto its hubs, we get a subset of 
			$\mathbb Z^\ell$ which is a union of slices.  These slices are not all isomorphic to faces. When $e=3$, the case we calculated most intensively, there was an infinite sequence of slices which are were isomorphic to a highest weight module for the Weyl group $\mathring{W}$, with dominant integral weights whose coefficients in $\mathbb{Z}^\ell$ were $(a,b)$, $(a+1,b+1)$,$(a+2,b+2),\dots$.
			isomorphic to faces and thus isomorphic to  highest weight modules.
			
			\item We conjecture that there is a duality between the canonical basis elements throughout the face and thus from a calculation point of view it is necessary to calculate the canonical basis element only for the upper half of the face. Under this duality, a natural basis element $|\mu \rangle$  with coefficient $v^a$ goes to a multipartition in which
			for each residue $i$ added from the top down to form $\mu$, a residue $t-i$ is removed from the 
			corresponding multipartition of $\rho$, with position chosen from the bottom up, and the multipartition is multiplied by $v^{(d-a)}$, where $d$ is the defect of the 
			weight of $\mu$. 
			
			\item The elements of a face all lie in $\max(\Lambda)$ since the content $c_j$ for $j \notin I_0$ is zero. If, in addition, $a_0=0$ and $I_0=I-\{0\}$, then the entire face is a fundamental region for $\max(\Lambda)$ under the action of the abelian subgroup $T$, and any canonical basis elements calculated for the face are readily calculated, with the same shape, throughout $\max(\Lambda)$.  Even when $a_0 \neq 0$, perhaps the canonical basis elements for the slices containing dominant integral weights can be readily calculated.  	
		\end{enumerate}
		
		For anyone interested in tackling one of these questions, we would be happy to provide the programs we use to calculate canonical basis elements and multipartitions, which are based on the SageMath \cite{SM} program ``fock\_space'' by Travis Scrimshaw. We do not intend to pursue any of these problems right now and believe that there is enough material here for another Ph.D. thesis.


\begin{thebibliography}{SchHap}
			\bibitem{AKT} S. Ariki, V. Kreiman, \& S. Tsuchioka, \textit{On the tensor product of two basic representations of $U_v(\hat sl_e)$}, Advances in Mathematics 218 (2008), 28-86.
			\bibitem {AS} H. Arisha \& M. Schaps \textit{Maximal Strings in the crystal graph of spin representations of symmetric and alternating groups}, Comm. in Alg., Vol 37, no. 11 (2009), 3779-3795.
			\bibitem{AOS1} O. Amara-Omari   \& M. Schaps. \textit{Unidirectional Littelmann paths for crystals of type A and rank 2}. Alg. Rep. Theory (2021). https://doi.org/10.1007/s10468-021-10063-9.
			\bibitem{AOS2} O. Amara-Omari   \& M. Schaps. \textit{Non-recursive canonical basis calculations for low-rank Kashiwara crystal of affine type A}. Rocky Mountain J. Math. 52(6): 1929-1955 (December 2022)
			\bibitem{AOS3} O. Amara-Omari   \& M. Schaps. \textit{External vertices for crystals of type A}.  arXiv 1811.11413. To appear in Alg. Rep. Theory.
			\bibitem {BFS} O. Barshavsky, M. Fayers \&  M. Schaps, \textit{ A non-recursive criterion for weights of highest-weight modules for affine Lie algebras}, Israel Jour. of Mathematics, vol.197(1) (2013), 237-261.
			\bibitem{B} N. Bourbaki, Lie Groups and Lie Algebras, Springer Verlag (1989).
			\bibitem{BK} J. Brundan and A. Kleshchev, \textit{Graded decomposition numbers for cyclotomic Hecke algebras}, Adv.\ Math. \textbf{222} (2009), 1883--1942.
			\bibitem{BS} D. Bump and A. Schilling. Crystal Bases, Representations and Combinatorics, World Scientific, Hackensack, NJ (2017).
			\bibitem{EE}
			H. Exton, \textit{$q$-Hypergeometric Functions and Applications}, New York: Halstead Press, Chichester: Ellis Horwood, 1983.
			\bibitem{Fa} M. Fayers, \textit{Weights of multipartitions and representations of Ariki--Koike algebras II}, Adv.\ Math.\ \textbf{206} (2006), 112--144.
			\bibitem{Fa2} M. Fayers, \textit{An LLT-type algorithm for calculating higher type canonical bases}, Journal of Pure and Applied Algebra, Volume 214, Issue 12, ( 2010), Pages 2186-2198
			\bibitem{FS} M. Fayers and M. Schaps, \textit{A non-recursive criterion for weights of highest-weight modules for affine Lie algebras}, arxiv 1002.3457v2.
			\bibitem{H} J. Humphreys, Lie Algebras and their Representations, 3rd ed. (1972) Springer Verlag.
			\bibitem{Ka} V. Kac, {Infinite Dimensional Lie Algebras}, 3rd ed., Cambridge University Press (1990).
			\bibitem{Kl} A. Kleshchev, \textit{Representation theory of symmetric groups and related Hecke algebras}, Bull.\ Amer.\ Math.\ Soc.\ \textbf{47} (2010), 419--481.
			\bibitem{LLT} A. Lascoux, B. Leclerc \& J.-Y. Thibon, \textit{Hecke algebras at roots of unity and crystal bases of quantum affine algebras}, Comm. Math. Phys. \textbf{181} (1996), 205-63.
			\bibitem{M1}A. Mathas, \textit{Simple modules of Ariki-Koike algebras}, \textit{Proc. Sym. Pure Math}(1997), 383-396.
			\bibitem{M2}A. Mathas, {Iwahori-Hecke Algebras and Schur Algebras of the Symmetric Group}, University Lecture Series, Vol. 15, \textit{American Mathematical Society}(1999).
			\bibitem{SM} SageMath, the Sage Mathematics Software System (Version 9.1),
			The Sage Developers, 2020, https://www.sagemath.org.
		\end{thebibliography}
	\end{document}